\documentclass[11pt]{article}

\usepackage{amssymb,amsmath,amsfonts}
\usepackage{graphicx,color,enumitem}
\usepackage{amsthm} 
\usepackage{bm}
\usepackage[round]{natbib}
\usepackage{geometry}

\usepackage[colorinlistoftodos, textwidth=4cm, shadow]{todonotes}
\RequirePackage[colorlinks,citecolor=blue,urlcolor=blue]{hyperref}

\usepackage{bbm}
\usepackage{capt-of}

\newcommand{\E}{{\mathbb E}}
\newcommand{\F}{{\mathbb F}}

\renewcommand{\P}{{\mathbb P}}
\newcommand{\Q}{{\mathbb Q}}
\newcommand{\C}{{\mathbb C}}
\newcommand{\R}{{\mathbb R}}

\newcommand{\N}{{\mathbb N}}

\newcommand{\Gcal}{{\mathcal G}}

\newcommand{\Mid}{{\ \Big|\ }}

\newcommand{\Fc}{{\mathcal F}}

\newtheorem{theorem}{Theorem}

\newtheorem{lemma}[theorem]{Lemma}

\newtheorem{remark}[theorem]{Remark}

\theoremstyle{definition}
\newtheorem{example}[theorem]{Example}

\numberwithin{equation}{section}
\numberwithin{theorem}{section}

\definecolor{darkgreen}{rgb}{0,0.7,0}

\newcommand{\iii}{{\vert\kern-0.25ex\vert\kern-0.25ex\vert}}

\usepackage{mathtools}
\mathtoolsset{showonlyrefs}

\newcommand{\RR}{\mathbb{R}}

\newcommand{\PP}{\mathbb{P}}
\newcommand{\EE}{\mathbb{E}}

\newcommand{\ve}{\varepsilon}

\begin{document}
	
	\title{Weak existence and uniqueness  for  affine stochastic Volterra equations with $L^1$-kernels}
	
	\author{Eduardo Abi Jaber\thanks{Universit\'e Paris 1 Panth\'eon-Sorbonne, Centre d'Economie de la Sorbonne, 106, Boulevard de l'H\^opital, 75013 Paris, eduardo.abi-jaber@univ-paris1.fr. I would like to thank Mathieu Rosenbaum for presenting me the uniqueness problem, and Ryan McCrickerd for interesting discussions. I am also grateful for the editor and two anonymous referees whose insightful comments and suggestions have significantly improved the manuscript.}}

	\maketitle
	
	\begin{abstract}
	We provide existence, uniqueness and stability results for  affine stochastic Volterra equations with $L^1$-kernels and  jumps. Such equations arise as scaling limits of branching processes in population genetics and	self-exciting Hawkes processes in mathematical finance.   The strategy we adopt for the existence part is  based on approximations using stochastic Volterra equations with $L^2$-kernels combined with a general stability result.  Most importantly,  we establish weak uniqueness using a  duality argument on the  Fourier--Laplace transform via a deterministic Riccati--Volterra integral equation. We illustrate the applicability of our results on Hawkes processes and a class of hyper-rough Volterra Heston models with a Hurst index $H \in (-1/2,1/2]$. 
 	\\[2ex] 
		\noindent{\textbf {Keywords:}} Stochastic Volterra equations, Affine Volterra processes, Riccati--Volterra equations, superprocesses, Hawkes processes, rough volatility.\\[2ex] 
		\noindent{\textbf {MSC2010 Classification:}}  	60H20, 	60G22, 	45D05
	\end{abstract}


\section{Introduction}
We establish weak existence, uniqueness and stability results for stochastic Volterra equation with locally $L^1$--kernels $K$ in the form
\begin{align}\label{eq:generalsve}
X_t &=  G_0(t) + \int_0^t K(t-s)  Z_s ds, \quad t \geq 0,
\end{align}
for a given function $G_0:\R_+ \to \R$ where $Z$ is a real-valued semimartingale, starting from zero,  with affine characteristics in $X$
\begin{align}\label{eq:charZ}
    (bX, c X,  {\nu(d\zeta)X}), 
\end{align} 
with $b \in \RR$, $c\geq 0$, $\nu$ a nonnegative measure on $\R_+$ such that $\int_{\R_+} \zeta^2 \nu(d\zeta)<\infty$, with respect to the `truncation function' $\chi(\zeta)=\zeta$.  For $L^2$--kernels this formulation was recently introduced in \citet{AJCPL:19}, where $Z$ is a semimartingale but whose characteristics are absolutely continuous with respect to the Lebesgue measure.  In the $L^1$ setting, $X$ may fail to be absolutely continuous with respect to the Lebesgue measure, as will be explained in the sequel. For this reason, our study falls beyond the scope of  \citet{AJCPL:19}.

Our motivation for studying such convolution equations is twofold. Stochastic Volterra equations with kernels that are locally in $L^1$ but not in $L^2$ with  $c>0$ $\nu\equiv 0$ arise as scaling limits of branching processes in population genetics and	self--exciting Hawkes processes in mathematical finance. As we will highlight  in the sequel, the $L^1$-framework allows  for instance to make sense of fractional dynamics, inspired by the fractional Brownian motion, for negative Hurst indices $H\in(-1/2,0)$.	

$\bullet$ \textit{From branching processes to stochastic Volterra equations.}  The link was formulated for the first time in  \citet{mytnik2015uniqueness} to motivate the study of stochastic Volterra equations with $L^2$--kernels. In the sequel we re-formulate the aforementioned introductory exposition linking super--processes with stochastic Volterra equations with $L^1$--kernels.  Consider a system  of $n$ \textit{reactant} particles  in one dimension moving independently according to a standard Brownian motion and branching only in the presence of a  \textit{catalyst}. The  \textit{catalyst} region  at a certain time $t$ is  defined as the support of some deterministic measure $\rho_t(dx)$. Whenever a particle enters in the \textit{catalyst} region and after spending a random time in the vicinity of the \textit{catalyst},  it  will either die or  split into two new particles, with equal probabilities. The measure $\rho_t(dx)$ determines the local branching rate in space and time depending on the location and the concentration of the \textit{catalyst}. Two typical examples are $\rho_t(dx) \equiv \bar \rho dx$  where the  branching occurs in the entire space with constant rate $\bar \rho$ and  $\rho_t(dx)=\delta_0(dx) $  for a  branching occurring  with infinite rate only  when the particle hits a highly concentrated single point \textit{catalyst} located at $0$. In case of branching, the two offspring particles evolve independently with the same spatial movement and branching mechanism as their parent. 

One can view the {reactant} as a rescaled measure-valued process $(\bar Y^n_t(dx))_{t \geq 0}$ defined by
$$ \bar Y^n_t(B) =  \frac{\mbox{number of particles in $B$ at time $t$}}{n}, \quad  \mbox{for every Borel set } B.$$
Sending the number of particles to infinity,  one can establish the convergence towards a measure-valued macroscopic \textit{reactant} $\bar Y$, coined \textit{catalytic super-Brownian motion}, which solves an infinite dimensional martingale problem, see \citet{dawson1991critical, etheridge2000introduction, perkins2002part} and the references therein. Moreover,  in the presence of a suitable deterministic \textit{catalyst}  $\rho=(\rho_t(dx))_{t \geq 0}$ having no atoms, the measure-valued process $\bar Y$  admits a density  $\bar Y_t (dx) = Y_t (x) dx $ solution to the following stochastic partial differential equation in mild form
\begin{align}\label{E:introspde2}
Y_t (x) =  \int_{\RR} p_t(x-y) Y_0(y)dy +   \int_{[0,t]\times\RR }  p_{t-s}(x-y)  \sqrt{Y_s(y)} W^{\rho} (ds, dy). 
\end{align} 
where  $Y_0$ is an input curve,  $p_t(x)= (2\pi t)^{-1/2}\exp(-x^2/(2t))$ is  the heat kernel
and ${W}^{\rho}$ is a space-time noise with covariance structure determined by  $\rho$, we refer to \cite{zahle2005space} for more details. The previous equation is only valid if $\rho$ has no atoms. One could still heuristically set $\rho_t(dx) = \delta_0(dx)$ in \eqref{E:introspde2} for the extreme case of a single point \textit{catalyst} at $0$, which would formally correspond to the \textit{catalytic super-Brownian motion} of \citet{dawson1994super}. Then, the space-time noise reduces to a standard Brownian motion $W$ so that  evaluation at $x=0$ yields 
\begin{align}\label{E:introspde3}
Y_t (0) =  g_0(t) +  \frac{1}{\sqrt{2 \pi}} \int_0^t   (t-s)^{-1/2}   dZ_s,
\end{align} 
where $dZ_t= \sqrt{Y_t(0)}dW_t$  and $g_0(t) =  \int_{\RR} p_t(y) Y_0(y)dy$.
 The link with stochastic Volterra equations of the form \eqref{eq:generalsve} is established by considering  the local occupation time at the catalyst point $0$  defined by
\begin{align}\label{eq:localtime}
X_t= \lim_{\ve \to 0} \int_0^t \int_{\RR} p^{\ve}(y) \bar Y_s(dy) \,ds,\quad  t \geq 0,
\end{align}
where $p^{\ve}$  is a suitable  smoothing kernel of the dirac point mass at $0$. Integrating both sides of equation \eqref{E:introspde3} with respect to time and formally interchanging the integrals lead to
\begin{align}
X_t &= \int_0^t Y_s(0)ds \label{eq:localtimevolterra0} \\
&= \int_0^t g_0(s)ds +  \frac{1}{\sqrt{2 \pi}} \int_0^t {(t-s)^{-1/2}} Z_s ds,\label{eq:localtimevolterra}
\end{align}
such that $Z$ is a continuous semimartingale with  affine characteristics $(0,X,0)$. Consequently, $X$ solves \eqref{eq:generalsve} for the kernel
\begin{align}\label{eq:kernel0}
 K_0 (t)= \frac{t^{-1/2}}{\sqrt{2 \pi}}, \quad t >0,
 \end{align}
 which is locally in $L^1$ but not in $L^2$. Needless to say, one is not allowed to plug the Dirac measure in \eqref{E:introspde2}. Indeed, in the presence of a single point \textit{catalyst},  the \textit{catalytic super-Brownian motion} does not admit a density at the \textit{catalyst} position as shown by   \cite{dawson1994super} and the identities  \eqref{E:introspde2} and \eqref{eq:localtimevolterra0} break down. The local occupation time $X$ is even singular with respect to the Lebesgue measure, see \citet{dawson1995singularity,fleischmann1995new}. Still, one can rigorously prove that the local occupation time $X$ defined by  \eqref{eq:localtime} solves \eqref{eq:localtimevolterra} by appealing to the martingale problem of the measure--valued process $\bar Y$, we refer to Appendix~\ref{A:localtime} for a  rigorous derivation.
 
 $\bullet$ \textit{From Hawkes processes to stochastic Volterra equations.} More recently,  
 for particular choices of $G_0$ and  kernels, solutions to \eqref{eq:generalsve} were obtained  in \citet{Jusselin_2018}  as scaling limits of Hawkes processes  $(N^n)_{n \geq 1}$ with  respective intensities 
  \begin{align} 
 \lambda^n_t = g_0^n(t) + \int_0^t K^n(t-s)dN^n_s, \quad t  \geq 0, 
 \end{align}
 for some suitable  function $g_0^n$ and kernel $K^n$. The rescaled sequence of integrated accelerated intensities $X^n= {\frac 1 n }\int_0^{\cdot} \lambda^n_{ns} ds$ is shown to converge to a continuous process $X$ satisfying \eqref{eq:generalsve}  for the fractional kernel\footnote{{To be more precise, in \citet{Jusselin_2018}, the limiting kernel is not the fractional kernel but the so-called Mittag-Leffler function  and the process $Z$ there has characteristics $(0,X,0)$.  This can be equivalently re-written with the fractional kernel modulo the addition of a suitable drift $bX$, we refer to Example~\ref{E:Hawkesapprox} below for more details.}} 
  \begin{align}\label{eq:kernelfrac}
 K_H(t)=\frac {t^{H-1/2}}{\Gamma(H+1/2)}, \quad t >0,
 \end{align}
 with $H \in (1/2,1/2]$. We note that for $H=0$ the fractional kernel reduces to \eqref{eq:kernel0}, up to a normalizing constant. In other words, when $H=0$, the scaling limit of the integrated intensities of Hawkes processes can be seen  as the local occupation time of the \textit{catalytic super-Brownian motion} of \citet{dawson1994super}, provided uniqueness holds. Similarly, when $H\leq 0$, $K_H$ lies locally in $L^1$ but not in $L^2$, and one can also show that in this case $X$ is not absolutely continuous with respect to the Lebesgue measure, see  \citet[Proposition 4.6]{Jusselin_2018}. For $H>0$, the kernel \eqref{eq:kernelfrac} is nothing else but the kernel that appears in the celebrated  \citet{mandelbrot1968fractional} decomposition of fractional Brownian motion $W^H$:
 \begin{align*}
 W^H_t = \int_0^t K_H(t-s)dW_s + \int_{-\infty}^0 (K_H(t-s)-K_H(-s))dW_s
 \end{align*}
 where $W$ is a two-sided standard Brownian motion and $H>0$ is required to make sense of the stochastic convolution with respect to Brownian motion in the $L^2$-theory of Kiyosi It\^o. In this sense, the $L^1$-framework allows for a generalization of fractional type dynamics with negative Hurst indices. 
 
 In both of the motivating cases, one can compute the Laplace transform of $X$, 
  modulo a deterministic Riccati--Volterra equation of the form 
 \begin{align*}
 \psi(t) = \int_0^t K_H(t-s)\left( \frac 1 2 \psi^2(s) -1 \right)ds,
 \end{align*}
 either by using 
 the dual process of the \textit{catalytic super-Brownian motion}, see  \citet[Equations (4.2.1)-(4.2.2)]{dawson1994super}, or by exploiting the affine structure of the   approximating  Hawkes processes, see \citet[Theorem 3.4]{Jusselin_2018}. Both constructions provide solutions to \eqref{eq:generalsve}, but do not yield uniqueness. Establishing weak uniqueness is one of the main motivation of this work.

In the present paper, we provide a generic treatment of  the limiting macroscopic equation \eqref{eq:generalsve} and we allow for (infinite activity) jumps in $Z$. For instance, Hawkes processes can be recovered by setting $c=0$ and $\nu=\delta_1$. The strategy we adopt is  based on approximations using stochastic Volterra equations with $L^2$ kernels, whose existence and uniqueness theory is now well--established, see \citet{AJCPL:19,Abi_Jaber_2017} and the references therein. By doing so, we avoid the infinite-dimensional analysis used for super-processes, we also circumvent the need to study scaling limits of  Hawkes processes, allowing for more generality in the choice of  kernels $K$ and input functions $G_0$. Along the way, we derive a general stability result that encompasses the motivating example with Hawkes processes. Most importantly,  we establish weak uniqueness using a  duality argument on the  Fourier--Laplace transform of $X$ via a deterministic Riccati--Volterra integral equation. In particular, this expression extends the one obtained for affine Volterra processes with $L^2$-kernels in \citet{Abi_Jaber_2017,cuchiero2020generalized}. We illustrate the applicability of our results on a class of hyper-rough Volterra Heston models with a Hurst index $H \in (-1/2,1/2]$ and  jumps  complementing the results of \citet{Abi_Jaber_2017,El_Euch_2019,Jusselin_2018}. Such models have recently known a growing interest  to account for rough volatility,   a universal phenomena observed in financial markets, see \citet{Gatheral_2018}. \vspace{0.4em}

\textbf{Notations} $\Delta_h$ stands for the shift operator, i.e.~$\Delta_h g=g(h+\cdot)$ and $dg$ is the distributional derivative of a right--continuous function $g$ with locally bounded variation, i.e.~$dg((s,t]) = g(t) -g(s)$. 
For a suitable Borel function $f$ the quantity $\int_{0}^{\cdot} f(s) dg(s)$  will stand for the Lebesgue--Stieltjes integral, whenever the integral exists. Similarly, for each $t<T$,  the convolution $\int_{0}^t f(t-s)dg(s)$ is defined as the Lebesgue--Stieltjes integral $\int_0^T \boldsymbol 1_{[0,t]} f(t-s) dg(s)$ whenever this latter quantity is well--defined.\vspace{0.4em}

\textbf{Outline}  Section~\ref{S:Main} states our main existence, uniqueness and stability  results together with the expression for the Fourier--Laplace transform.  Section~\ref{S:apriori} provides a-priori estimates for the solution. In Section~\ref{S:stability}, we derive a general stability results for stochastic Volterra equations with $L^1$--kernels.  These results are used to establish weak existence for the stochastic Volterra equation in Section~\ref{S:existence}. Furthermore, an existence result for Riccati--Volterra equations with  $L^1$--kernels is derived there.   Weak uniqueness is then established  by completely characterizing the Fourier--Laplace transform of the solution in terms of the  Riccati--Volterra equation of Section \ref{S:uniqueness}. In Section~\ref{S:Heston}, we apply our results to obtain existence, uniqueness and the characteristic function of the log-price in  hyper--rough Volterra Heston models.  Finally, we provide a more rigorous derivation of the stochastic Volterra equation satisfied by the local occupation time of the catalytic super--Brownian motion  in  Appendix~\ref{A:localtime}.

\section{Main results}\label{S:Main}
In this section, we present our main results together with the strategy we adopt. We start by making precise the concept of solution. 

We call $X$ a  weak solution to \eqref{eq:generalsve} for the input $(G_0,K,b,c,\nu)$, if there exists  a  filtered probability space $(\Omega,\mathcal F,(\mathcal F_t)_{t \geq 0}, \PP)$ supporting a non-decreasing, nonnegative, continuous and adapted process $X$ and a  semimartingale $Z$ whose characteristics are given by \eqref{eq:charZ} such that \eqref{eq:generalsve} holds $\PP$--almost surely. We stress that a weak solution is given by the triplet $(X,(\Omega,\mathcal F,(\mathcal F_t)_{t \geq 0}, \PP), Z )$. To ease notations we simply denote  the solution by $X$.  In this case,
$Z$ admits the following  decomposition 
\begin{align}\label{eq:decompZ}
Z_t = b X_t + M^c_t + M^d_t, \quad  t \geq 0,
\end{align}
where $M^c$ is a continuous local martingale of quadratic variation $cX$ and $M^d$ is a purely discontinuous local martingale given by $\int_{[0,t]\times \R_+} \zeta \left(\mu^Z(dt,d\zeta) - \nu(d\zeta)dX_t \right)$ and $\mu^Z$ is the jump measure of $Z$, we refer to \citet[II.2.5-6]{Jacod_2003}.  We say that weak uniqueness holds for the inputs $(G_0,K,b,c,\nu)$ if given two weak solutions   $(X,(\Omega,\mathcal F,(\mathcal F_t)_{t \geq 0}, \PP)  , Z ) $ and $(X', (\Omega',\mathcal F',(\mathcal F'_t)_{t \geq 0}, \PP'),Z')$, $X$ and $X'$ have the same finite dimensional marginals.

 One first notes that the formulation \eqref{eq:generalsve}, when $\nu\equiv 0$,   differs from the one given in \citet[Equation (3)]{Jusselin_2018}, where 
\begin{align*}
X_t= G_0(t) + \int_0^t \left(\int_0^{t-s}K(r) dr\right)  d Z_{s}.
\end{align*} 
Although these two formulations are equivalent, thanks to a stochastic Fubini theorem, the advantages of considering the formulation \eqref{eq:generalsve}  as starting point,  which is inspired by the `martingale problem' formulation of stochastic Volterra equations recently introduced in \citet{AJCPL:19}  will  become clear in the sequel.

The following lemma  establishes the link with stochastic Volterra equations with $L^2$--kernels, as the one studied for instance in \citet{AJCPL:19,Abi_Jaber_2017,cuchiero2020generalized}.
\begin{lemma}\label{L:spotvariance}
	Fix  $K \in L^2_{\rm loc}(\R_+,\R)$ and $g_0 \in L^1_{\rm loc}(\R_+,\R)$. Assume that there exists a non-decreasing continuous adapted process $X$ on some filtered probability space $(\Omega,\mathcal F,(\mathcal F_t)_{t \geq 0}, \PP)$ such that   \begin{align}\label{eq:barVVolterra} 
	X_t = \int_0^t g_0(s)ds + \int_0^t K(t-s) Z_s ds,
	\end{align} 
	 with  $Z$ given by \eqref{eq:decompZ}.
	Then,  
	$X = \int_0^{\cdot} Y_s ds$ where $Y$ is a nonnegative weak solution to the following stochastic Volterra equation
	\begin{align}\label{eq:svel2}
	Y_t = g_0(t) + \int_0^t K(t-s) dZ_s, \quad  \PP\otimes dt-a.e.
	\end{align}
	where the differential characteristics of $Z$ with respect to the Lebesgue measure are given by $(bY,cY,\nu(d\zeta)Y)$.\\
	Conversely, assume there exists a nonnegative weak solution $Y$ to the stochastic Volterra equation \eqref{eq:svel2} such that $\int_0^T Y_s ds<\infty$ for all $T>0$, then $X=\int_0^{\cdot} Y_s ds$ is a continuous non-decreasing solution to \eqref{eq:barVVolterra}.
\end{lemma}

\begin{proof}
	Fix $t \geq 0$. An application of a stochastic Fubini theorem, see \citet[Lemma 3.2]{AJCPL:19},   yields
	\begin{align*}
	\int_0^t K(t-s) Z_s ds &= \int_0^t K(s)  \left(  \int_0^{t-s}  dZ_s \right) ds \\
	&= \int_0^t  \left(  \int_0^{t-r}  K(s) ds \right)  dZ_r \\
	&= \int_0^t  \left(  \int_0^{t}  K(s-r) \mathbf 1_{\{r \leq s \}} ds \right) dZ_r \\
	&= \int_0^t  \left(  \int_0^{s}  K(s-r)  dZ_r \right)   ds.
	\end{align*}
	Thus, $X$ admits a density $Y$  with respect to the Lebesgue measure, such that 
	\begin{align*}
	Y_t = g_0(t) + \int_0^t K(t-r) dZ_r,
	\end{align*}
	and the characteristics of $Z$ read
	 $$\left(b\int_0^{\cdot} Y_s ds , c\int_0^{\cdot} Y_s ds, \int_{[0,\cdot]\times \R_+}   Y_s ds\nu(d\zeta)  \right).$$
	Since $X$  is non-decreasing almost surely, $Y$ is nonnegative $\PP \otimes dt$. 
	 The claimed stochastic Volterra equation \eqref{eq:svel2} readily follows. The  converse direction follows along the same lines by integrating both sides of \eqref{eq:svel2} and applying a stochastic Fubini theorem as above to get \eqref{eq:barVVolterra}. 
\end{proof}

\subsection{Uniqueness and Fourier--Laplace transform}
{We start by stating our first main result concerning the representation of the Fourier--Laplace transform of the joint process $(X,M^c,M^d)$ and the weak uniqueness statement for \eqref{eq:generalsve} in terms of a solution to the  Riccati--Volterra equation
		{\begin{align}
		\psi(t) &= \int_0^t K(t-s)F(s,\psi(s))ds,  \label{eq:RiccatiL2a}\\
		F(s,u) &=  f_0(s) + \frac 1 2 cf^2_1(s)  + \left(b+cf_1(s) \right)u + \frac c 2 u^2 \nonumber 
		\\ &\quad\quad  + \int_{\R_+} \left(e^{(f_2(s)+u)\zeta}-1 -(f_2(s)+u)\zeta \right)\nu(d\zeta), \label{eq:RiccatiL2b}
		\end{align}
where $f_0,f_1,f_2:\R_+ \mapsto \mathbb B$ are  suitable functions.	
}
We introduce  the following process which enters in the representation of the Fourier--Laplace transform:
\begin{align}
G_t(s) &= G_0(s) + \int_t^{s} g_t(u) du, \quad   t\leq s, \label{eq:Gt} \\
g_t(u) &= \int_0^t K(u-r)dZ_r, \quad  \quad  \,\, t<u.\label{eq:gt}
\end{align}
The stochastic convolution  $\int_0^t K(s-r)dZ_r= \int_0^t \Delta_{s-t} K(t-r)dZ_r$ is well-defined as an It\^o integral, for all $s>t$, provided that  the  shifted kernels $\Delta_h K:=K(\cdot+h)$ are in $L^2_{\rm loc}(\R_+,\R)$ for any $h>0$.

	\begin{theorem}\label{T:MainIntrouniqueness} Fix $b\in \R,c\geq 0$ and $\nu$ a nonnegative measure supported on $\R_+$ such that $\int_{\R_+}\zeta^2 \nu(d\zeta)<\infty$.
		Fix a kernel  $K \in L^1_{\rm loc}(\R_+,\R)$ such that its shifted kernels $\Delta_h K=K(\cdot + h)$ are in $L^2_{\rm loc}(\R_+,\R)$ for any $h>0$,  and a non-decreasing continuous  function $G_0:\R_+\mapsto \R$.	Fix $T\geq 0$ and  three continuous functions $f_0,f_1,f_2:[0,T]\mapsto \C$. 
		 Assume that  there exists a solution $\psi\in C([0,T],\mathbb C)$  to the Riccati--Volterra equation \eqref{eq:RiccatiL2a}-\eqref{eq:RiccatiL2b}	such that 
		\begin{align}\label{eq:condmomentpsi}
		\sup_{t\leq T} \int_{\R_+} e^{\Re(f_2(t)+\psi(t)) \zeta} \zeta^2 \nu(d\zeta) < \infty 
		\end{align}
Then, for any weak solution $X$ of \eqref{eq:generalsve}, the joint conditional Fourier--Laplace transform of
$$ R_{t,T}=  \int_t^T \! f_0(T-s)dX_s + \int_t^T \! f_1(T-s)dM_s^c + \int_t^T \! f_2(T-s)dM_s^d,$$
 {where $M^c$ and $M^d$ are the local martingales appearing in \eqref{eq:decompZ},}  is given by 
		\begin{align}
		&\EE\left[ \exp\left( R_{t,T} \right) \bigg| \mathcal F_t \right] = \exp\left(\int_t^T \! F(T-s, \psi(T-s))dG_t(s)\right), \quad t\leq T, \label{eq:affinetransform}
		\end{align}
	 where $G_t$ is defined as in \eqref{eq:Gt}. In particular, if the Riccati--Volterra equation  \eqref{eq:RiccatiL2a}-\eqref{eq:RiccatiL2b} with $f_2=f_3\equiv 0$, admits a continuous solution  $\psi$ such that $\sup_{t\leq T}\int_{\R_+} e^{\Re(\psi(t))\zeta}\zeta^2\nu(d\zeta)<\infty$ for any continuous function $f_0:[0,T]\mapsto i\R$ and any $T\geq 0$, then weak uniqueness holds for \eqref{eq:generalsve} for the input $(G_0,K,b,c,\nu)$. 
	\end{theorem}
	
	\begin{proof} 
The proof is detailed in Section~\ref{S:uniqueness}.
	\end{proof}

The following example illustrates the applicability of Theorem~\ref{T:MainIntrouniqueness} in the case of pure jump Hawkes processes.  The example will be continued in Example~\ref{E:Hawkesapprox} to illustrate the scaling limits of Hawkes processes mentioned in the introduction. Section~\ref{S:Heston} provides another example of application of Theorem~\ref{T:MainIntrouniqueness} in the context of rough volatility modeling.  
\begin{example}\label{E:hawkeschar}
	Let $N$ denote a counting process with instantaneous intensity 
	\begin{align*}
	\lambda_t = g_0(s)+ \int_0^t K(t-s)dN_s,
	\end{align*}
	for some $g_0, K\in L^1_{\rm loc}(\R_+,\R_+)$. We are interested in the computation of the joint Fourier--Laplace transform of $(\lambda,N)$, more precisely of the quantity:
	\begin{align}
	\int_0^T h_0(T-s)\lambda_s ds + \int_0^T h_2(T-s)dN_s, 
	\end{align}
	for some continuous functions $h_0,h_2:\R_+ \mapsto \C$. 
	By an application of Lemma~\ref{L:spotvariance}, the integrated intensity $X=\int_0^{\cdot} \lambda_sds$ solves \eqref{eq:generalsve} with $G_0=\int_0^{\cdot}g_0(s)ds$ and $Z=N$ with affine characteristics in \eqref{eq:charZ}  given by 
	$$ (X, 0 , \delta_1(d\zeta) X),$$ 
	meaning that $(b,c,\nu)=(1,0,\delta_1)$ and $M^d=(N-X)$. 
Under the assumptions of Theorem~\ref{T:MainIntrouniqueness}, 	setting $f_0=h_0-h_2$, $f_1\equiv0$ and $f_2=h_2$ 
the joint Fourier-Laplace transform of $(\Lambda, N)$ is given by 
\begin{align}
&\EE\left[ \exp\left( \int_0^T h_0(T-s)\lambda_s ds + \int_0^T h_2(T-s)dN_s \right) \right] = \exp\left(\int_0^T \! F(T-s, \psi(T-s))g_0(s)ds\right) 
\end{align}
	where   the Riccati--Volterra equations \eqref{eq:RiccatiL2a}--\eqref{eq:RiccatiL2b} read 
	\begin{align}
	\psi(t) &= \int_0^t K(t-s)F(s,\psi(s))ds,  \quad t \leq T,\label{eq:RiccatiLhawkesa}\\
	F(s,u) &=  h_0(s)  +  e^{(h_2(s)+u)}-1 . \label{eq:Riccatihawkesb}
	\end{align}
	We refer to Theorem~\ref{T:existenceRiccatiL1} and Remark~\ref{R:riccaticond} below for the existence of such $\psi$.  
In particular, setting $h_0\equiv0$ and $h_2\equiv ia$ for some $a\in \R$, we recover the formula of \citet[Theorem 2]{hawkes1974cluster} derived using cluster representations of Hawkes processes, see also  \citet[Theorem 3.1]{El_Euch_2019}. 
\end{example}

Under additional assumptions on $K$ we prove the existence of a solution to the Riccati--Volterra equations \eqref{eq:RiccatiL2a}--\eqref{eq:RiccatiL2b}. For this we recall the notion of the resolvent of the first kind of a kernel: a measure $L$ of locally bounded variation is called  a resolvent of the first kind of the kernel $K \in L^1_{\rm loc}(\R_+,\R)$ if 
\begin{align}\label{eq:resolventL}
\int_{[0,t]} K(t-s)L(ds)= 1,
\quad t \geq 0. 
\end{align}  
If such $L$ exists, then it   is unique by \citet[Theorem 5.2.2]{gripenberg1990volterra}. We consider the following condition on the kernel $K$:
\begin{equation} \label{eq:K orthant}
\begin{minipage}[c][5em][c]{.8\textwidth}
\begin{center}
the kernel is nonnegative, non-increasing and {continuously differentiable} on $(0,\infty)$, and its resolvent of the first kind $L$ is nonnegative and non-increasing in the sense that $s\mapsto L([s,s+t])$ is non-increasing for all $t\ge0$.
\end{center}
\end{minipage}
\end{equation}

We note in \eqref{eq:K orthant} that any nonnegative and non-increasing kernel that is not identically zero admits a resolvent of the first kind; see \citet[Theorem~5.5.5]{gripenberg1990volterra}. The following example provides a large class of kernels for which \eqref{eq:K orthant} is satisfied. 
\begin{example} \label{E:CM}	
	If $K$ is completely monotone on $(0,\infty)$, then \eqref{eq:K orthant} holds due to \citet[Theorem~5.5.4]{gripenberg1990volterra}.  Recall that a function $f$ is called completely monotone on $(0,\infty)$ if it is infinitely differentiable there with $(-1)^k f^{(k)}(t) \geq 0$ for all $t>0$ and $k\geq 0$. We also note  that, for each $h >0$, the shifted kernel $\Delta_{h} K$ are  again completely monotone on $[0,\infty)$ so that \eqref{eq:K orthant} holds also for $\Delta_{h} K$. In particular, $\Delta_h K \in L^2_{\rm loc}(\R_+,\R)$, for each $h>0$.  This covers, for instance, any constant positive kernel, the fractional kernel $t^{H-1/2}$ with $H\in(-1/2,1/2]$, and the exponentially decaying kernel ${\rm e}^{-\eta t}$ with $\eta>0$. Moreover,  sums and products of completely monotone kernels are completely monotone.	By combining the above examples we find that the Gamma kernel $K(t)=t^{H-1/2}{\rm e}^{-\eta t}$ for $H \in (-1/2,1/2]$ and $\eta\ge0$ satisfies~\eqref{eq:K orthant}. 
\end{example}

The following theorem establishes the existence of solutions to the Riccati--Volterra equation \eqref{eq:RiccatiL2a}-\eqref{eq:RiccatiL2b} under structural assumptions on $f_0,f_1,f_2$.  

\begin{theorem}\label{T:existenceRiccatiL1} Fix $b\in \R,c\geq 0$ and $\nu$ a nonnegative measure supported on $\R_+$ such that $\int_{\R_+}\zeta^2 \nu(d\zeta)<\infty$. Let $f_0,f_1,f_2:\R_+ \mapsto \C$ be three continuous functions  
	such that 
	\begin{align}\label{eq:coeffaf} 
	\Re(f_0) + \frac c 2 \Re(f_1)^2 + \frac{1} 2 \int_{\R+} \zeta^2 \nu(d\zeta) \Re(f_2)^2 \leq 0 \quad \mbox{and} \quad  \Re(f_2) \leq 0.
	\end{align}
	Fix a  kernel $K \in L^1_{\rm loc}(\R_+,\R)$ satisfying \eqref{eq:K orthant}. Then, there exists a continuous solution $\psi$ to \eqref{eq:RiccatiL2a}-\eqref{eq:RiccatiL2b} such that $\Re(\psi(t))\leq 0$, for all $t\geq 0$.
\end{theorem}
\begin{proof}
	We refer to Section~\ref{S:existenceRiccati}.
\end{proof}
}

\begin{remark}\label{R:riccaticond}
\begin{itemize}
	\item 
Under \eqref{eq:coeffaf},  \eqref{eq:condmomentpsi} follows from the fact that $\Re(\psi)\leq 0$ as stated in Theorem~\ref{T:existenceRiccatiL1}.    
\item 
The condition \eqref{eq:coeffaf} is satisfied for instance if $\Re(f_0)\leq 0$ and $f_1,f_2:\R_+ \mapsto i\R$.  In particular, the second part of Theorem~\ref{T:MainIntrouniqueness} provides weak uniqueness for \eqref{eq:generalsve}.   Going back to Example~\ref{E:hawkeschar}, the existence of a solution to the corresponding Riccati--Volterra equation  is ensured provided that $\Re(h_0)\leq 0$ and $h_2:\R_+\mapsto i\R $.    
\end{itemize}
	\end{remark}

\begin{remark}
	If $K$ is in $L^2_{\rm loc}(\R_+,\R)$, then Theorems~\ref{T:MainIntrouniqueness}  and \ref{T:existenceRiccatiL1} agree with \citet[Theorem 5.12]{cuchiero2020generalized}  for the jump case;  if in addition $\nu \equiv 0$, then one recovers  \citet[Theorem 7.1]{Abi_Jaber_2017} and \citet[Theorem 2.3]{AJEE:19markovian} for the continuous case.
\end{remark}

\subsection{Stability and existence}

We now present our existence and stability results for solutions to \eqref{eq:generalsve}.
Our strategy for constructing solutions  with $L^1$-kernels relies on an approximation argument using $L^2$-kernels combined with Lemma~\ref{L:spotvariance}. 
To fix ideas, set $\nu\equiv 0$ and assume  $G_0=\lim_{n \to \infty} G_0^n$, with $G^n_0 = \int_0^{\cdot}g^n_0(s)ds$, for some sequence of $L^1_{\rm loc}$-functions $(g^n_0)_{n \geq 1}$.
Starting from a  $L^1_{\rm loc}$--kernel $K$, assume that there exists a sequence of $L^2_{\rm loc}$--kernels  $(K^n)_{n \geq 1}$ such that 
$$  K^n \to K, \quad \mbox{in }  L^1_{\rm loc}, \quad \mbox{as $n \to \infty$}.$$
Then, for each $n \geq 1$, $K^n$ being locally square--integrable, under suitable conditions on $(g_0^n,K^n)$,  the results in \citet{Abi_Jaber_2017,AJEE:19markovian} provide existence of nonnegative solution $Y^n$ for \eqref{eq:svel2} with $(g_0,K)$ replaced by $(g_0^n,K^n)$.  Setting $X^n= \int_0^{\cdot} Y^n_s ds $, Lemma \ref{L:spotvariance} provides a solution $X^n$ to \eqref{eq:generalsve} for the input $(G_0^n,K^n)$, that is 
\begin{align}\label{eq:Xnintro}
X^n_t = G_0^n(t) + \int_0^t K^n(t-s)Z^n_s ds, 
\end{align}
where the characteristics of $Z^n$ are $(bX^n,cX^n,0)$. Provided that   $(X^n)_{n \geq 1}$ is tight, it will admit a convergent subsequence towards a limiting process $X$. Finally, sending $n\to \infty$, one would expect  $X$ to solve \eqref{eq:generalsve} for $\nu \equiv 0$. 

We will adapt the same strategy in the case of jumps. Before this, we state our generic stability result.
\begin{theorem}\label{T:stability}   Assume that there exist sequences of  coefficients $(b^n,c^n,\nu^n(d\zeta))_{n \geq 1}$ with $\int_{\R}\zeta^2 \nu^n(d\zeta)<\infty$,  non-increasing kernels $(K^n)_{n\geq 1}$ and functions $(G_0^n)_{n \geq 1}$   such that
	\begin{enumerate}
		\item\label{eq:stabilityi} 
		$$b^n\to b, \quad  c^n + \int_{\R} \zeta^2 \nu^n(d\zeta)\to  c + \int_{\R} \zeta^2 \nu(d\zeta) \; \mbox{and} \; \int_{\R}h(\zeta)\nu^n(d\zeta)\to \int_{\R}h(\zeta)\nu(d\zeta),$$
		for any continuous and bounded function $h$ vanishing around zero,    as $n\to \infty$, for some $b\in \R$, $c \geq 0$ and $\nu(d\zeta)$ a nonnegative  measure  such that $\int_{\R} \zeta^2 \nu(d\zeta)< \infty$.
		\item \label{eq:stabilityii}
		$   \int_0^t| K^n(s)-K(s)|ds \to 0$, $\mbox{as $n\to \infty$}$, $t \geq 0$, for some non-increasing kernel $K \in L^1_{\rm loc}(\R_+,\R)$.
		\item \label{eq:stabilityiv}
		$\sup_{t\leq T} |G_0^n(t) - G_0(t)| \to 0, \; \mbox{as $n\to \infty$}, \;  T \geq 0$, for some continuous function $G_0$.
	\end{enumerate}
	Then, any   sequence of continuous nonnegative and non-decreasing solutions  $(X^n)_{n \geq 1}$ to \eqref{eq:generalsve} for the respective inputs $(G_0^n,K^n,b^n,c^n,\nu^n)$, is tight on the space of continuous functions $C([0,T],\RR)$ endowed with the uniform topology, for each $T >0$. Furthermore, any limit point $X$ is a continuous non-decreasing solution to  \eqref{eq:generalsve} for the input $(G_0,K,b,c,\nu)$.
\end{theorem}

\begin{proof}
	We refer to Section~\ref{S:stability}.
\end{proof}

\begin{remark}
	One notes that if $K$ satisfies \eqref{eq:K orthant}, then weak uniqueness holds thanks to Theorems~\ref{T:MainIntrouniqueness} and \ref{T:existenceRiccatiL1} so that one gets from Theorem~\ref{T:stability} that  $X^n\Rightarrow X$, where $X$ is the unique solution to \eqref{eq:generalsve}.
\end{remark}

 The following example illustrates an application of Theorem~\ref{T:stability} to the scaling limits of Hawkes processes mentioned in the introduction. We stress that the convergence of the second modified characteristic in Theorem~\ref{T:stability}-\ref{eq:stabilityi} allows to obtain continuous limiting semimartingales $Z$ from a sequence of jump semimartingales $Z^n$.  We recall that the resolvent of the second kind $R$ of $K$ is the unique $L^1_{\rm loc}$ function solution to 
 \begin{align}\label{eq:resolventeq}
 R(t) = K(t) +    \int_0^t K(t-s)R(s)ds = K(t) +    \int_0^t R(t-s) K(s)ds, \quad t \geq 0.
 \end{align}
 The resolvent $R$ exists, for any kernel $K \in L^1_{\rm loc}$, see \citet[Theorems 2.3.1  and 2.3.5]{gripenberg1990volterra}.
\begin{example}\label{E:Hawkesapprox}
	Fix $n\geq 1$ and a sequence of counting processes $(N^n)_{n \geq 1}$ with  respective intensities 
	\begin{align} 
	\lambda^n_t = g_0^n(t) + \int_0^t K^n(t-s)dN^n_s, \quad t  \geq 0, 
	\end{align}
	for some continuous function $g_0^n$ and kernel $K^n \in L^1_{\rm loc}$. Then, convolving both sides of the equation with  the resolvent of the second kind $R^n$ of $K^n$ and using \eqref{eq:resolventeq}, leads to 
	\begin{align} 
	\lambda^n_t = g_0^n(t) + \int_0^t R^n(t-s) g_0^n(s)ds + \int_0^t R^n(t-s)dM^n_s, \quad t  \geq 0, 
	\end{align}
	where $M^n=N^n-\int_0^{\cdot} \lambda^n_s ds$.  Let $X^n$ denote the rescaled sequence of integrated accelerated intensities $X^n= \frac 1 n \int_0^{\cdot} \lambda^n_{ns} ds$. Then, by an application of Lemma~\ref{L:spotvariance}, $X^n$ satisfies 
	\begin{align}
	X^n_t = G_0^n(t) + \int_0^t \tilde R^n(t-s)Z^n_s ds,
	\end{align}
	with $G^n_0(t)=\int_0^t \left(\frac 1 n g_0^n(ns) + \int_0^s \tilde R^n(s-u)g_0^n(nu)du \right)ds$, $\tilde R^n(t)=R^n(nt)$ and $Z^n_t=\frac 1 n M^n_{nt}$. Whence $Z^n$ is a pure jump martingale with jump sizes $\frac 1 n$ and integrated intensity $\int_0^{nt}\lambda^n_s ds = n^2 X^n_t $ so that its characteristics read $(0,0,n^2X^n \delta_{1/n}(d\zeta))$. Set $(b^n,c^n,\nu^n)=(0,0,n^2\delta_{1/n})$ and observe that 
	$$ \int_{\R} \zeta^2 \nu^n(d\zeta) = 1 \quad  \mbox{ and } \quad  \int_{\R} h(z)\nu^n(d\zeta)=h(1/n)n^2 \to 0,  \quad \mbox{as $n \to \infty$,} $$ 
	for any continuous and bounded function $h$ vanishing around 0. Whence, Theorem~\ref{T:stability}-\ref{eq:stabilityi} is satisfied with $(b,c,\nu)=(0,1,0)$. Provided that  $g_0^n$ and $K^n$  are chosen such that $(\tilde R^n)_{n\geq 1}$  and $(G_0^n)_{n\geq 1}$ satisfy  Theorem~\ref{T:stability}-\ref{eq:stabilityii}-\ref{eq:stabilityiv} for some $\tilde R$ and $G_0$, Theorem~\ref{T:stability}  yields that $(X^n)_{n\geq 1}$ is tight on $C([0,T],\R)$ such that any limit point satisfies 
	$$ X_t = G_0(t)+ \int_0^t \tilde R(t-s)Z_sds $$ 
	with $Z$ a continuous martingale with characteristics $(0,X,0)$. For instance, setting $g_0^n \equiv 0$,  \citet{Jusselin_2018} construct a sequence of kernels $K^n$ such that $\tilde R^n $ converges in $L^1([0,T],\R)$ towards the  function 
	$$ \tilde R_H(t)=\lambda t^{\alpha-1}E_{\alpha,\alpha}(-\lambda t^{\alpha}),$$
	for some $\lambda \in \R$, $\alpha=H+1/2$ with $H\in(-1/2,1/2)$ and $E_{\alpha,\alpha}(z)=\sum_{n\geq 0}z^n/\Gamma(\alpha(n+1))$ the so-called Mittag-Leffler function. In particular, $ \tilde R_H$  is the resolvent of the second kind of  $\lambda K_H$ where $K_H$ is the  fractional kernel given by \eqref{eq:kernelfrac}. In this case, using again the resolvent equation \eqref{eq:resolventeq}, one obtains 
	$$ X_t = \int_0^t \lambda K_H(t-s)\tilde Z_s ds,$$ with $\tilde Z= X + Z$.     	
\end{example}

\begin{remark}
	In the $L^2$ setting, i.e.~when the characteristics of $Z$ are absolutely continuous with respect to the Lebesgue measure, \citet[Theorem 1.6]{AJCPL:19} provides a generic stability result for stochastic Volterra equations with jumps using a `martingale problem' formulation, for general coefficients for the differential characteristics of $Z$ going beyond the affine case.   
\end{remark}

We now introduce the monotonicity and continuity  assumptions  needed on the kernel $K$ and the input function $G_0$ to construct non-decreasing and nonnegative solutions to \eqref{eq:generalsve}.   We assume that $K \in L^1_{\rm loc}(\R_+,\R)$ such that \eqref{eq:K orthant} holds.  Concerning the input function $G_0$, in the absence of jumps and for $K\in L^2_{\rm loc}$ \citet{AJEE:19markovian} provide  a set $\mathcal G_{K}$ of  admissible  input curves $g_0$ defined in terms of the resolvent of the first kind  $L$ to ensure the existence of non-negative solution for \eqref{eq:svel2}. To guarantee that the approximate solutions \eqref{eq:Xnintro} are non-decreasing,  we consider similarly to  \citet[Equations (2.4)--(2.5)]{AJEE:19markovian}, the following condition\footnote{{Under  \eqref{eq:K orthant} one can show that   $\Delta_hK*L$  is right-continuous and of locally bounded variation (see \citet[Remark B.3]{AJEE:19markovian}), thus the associated measure $d(\Delta_hK*L)$ is well defined.}}
\begin{align}\label{eq:notsofriendlycondition}
\Delta_h g_0 - (\Delta_hK * L)(0) g_0 - d(\Delta_hK * L) * g_0 \geq 0, \quad  h \geq 0,
\end{align} 
where we used the notation $(f*\mu)(t)=\int_0^t f(t-s)\mu(ds)$ for a measure of locally bounded variation $\mu$ and a function $f \in L^1_{\rm loc}$.   We then define  the space of admissible input curves $ \mathcal G_{K}$ to be
\begin{align}\label{eq:notsofriendlyset}
\mathcal G_{K} = \left\{ g_0 \mbox{ continuous,  satisfying } \eqref{eq:notsofriendlycondition} \mbox{ such that } g_0(0)\geq 0 \right\}.
\end{align} 
Two notable examples of such admissible input curves are:
\begin{example}\label{E:g0}
\begin{enumerate}
	\item \label{E:g0i}
	$g_0 \mbox{ continuous and non-decreasing with }  g_0(0) \geq 0,$
	\item \label{E:g0ii}
	$g_0(t) = x_0 +  \int_0^t K(t-s)\theta(s)  ds,$ for  some $x_0\geq 0$ and $\theta:\RR_+ \to \RR_+$ locally bounded,
\end{enumerate}
see e.g.~\citet[Example (2.2)]{AJEE:19markovian}.	
\end{example}

We are now in place to state the main existence (and uniqueness) result. 
\begin{theorem}\label{T:MainIntro}
	 Fix $b\in \R,c\geq 0$ and $\nu$ a nonnegative measure supported on $\R_+$ such that $\int_{\R_+}\zeta^2 \nu(d\zeta)<\infty$.
	Fix a  kernel $K \in L^1_{\rm loc}(\R_+,\R)$. Assume that  $K$ and its   shifted kernels $\Delta_{1/n} K$ satisfy  \eqref{eq:K orthant}, for any $n\geq 1$. Let $G_0=\lim_{n\to \infty}\int_0^{\cdot} g^{n}_0(s)ds$ for some functions $g^n_0 \in \mathcal{G}_{\Delta_{1/n} K}$, $n\geq 1$, and assume that $G_0$ is continuous. Then, there exists   a {unique} non-decreasing nonnegative continuous weak solution $X$ to \eqref{eq:generalsve} for the input $(G_0,K,b,c,\nu)$. \end{theorem}

\begin{proof} 
 		The proof for the existence part is given in~Section \ref{S:existencestochastic}.  The uniqueness statement is obtained from  Theorems~\ref{T:MainIntrouniqueness} and \ref{T:existenceRiccatiL1}, recall Remark~\ref{R:riccaticond}.
 	\end{proof}

\begin{remark}\label{R:assumptiong0}
	If    $K$ satisfies \eqref{eq:K orthant} and  $g_0 \in \mathcal G_K$ as in Example~\ref{E:g0}, the continuous function $G_0(t)=\int_0^t g_0(s)ds$ satisfies the assumption of Theorem~\ref{T:MainIntro}. Indeed:
	\begin{enumerate}
		\item
		Take $g_0$ as in Example~\ref{E:g0}-\ref{E:g0i}, fix  $n\geq1$ and set $g_0^n=\Delta_{1/n}g_0$. Clearly, $g_0^n$ satisfies again Example~\ref{E:g0}-\ref{E:g0i}, so that $g_0^n \in \mathcal G_{\Delta_{1/n}K}$.    Furthermore, we have that $\lim_{n\to \infty}\int_0^{t}g_0^n(s)ds=G_0(t)$ by virtue of  \citet[Lemma 4.3]{brezis2010functional}.
		\item
		Take $g_0$ as in Example~\ref{E:g0}-\ref{E:g0ii} and set $g_0^n(t)=x_0 + \int_0^t \Delta_{1/n}K(t-s)\theta(s)ds$, then, similarly we have that $\lim_{n\to \infty}\int_0^{t}g_0^n(s)ds=G_0(t)$ and $g_0^n\in \mathcal G_{\Delta_{1/n}K}$.
	\end{enumerate}
	\end{remark}

\section{A-priori estimates}\label{S:apriori}
We first provide a-priori estimates for solutions to \eqref{eq:generalsve}. We make use of the resolvent of the second kind $R$ of $K$ given in \eqref{eq:resolventeq}.

\begin{lemma}\label{L:apriori} Fix $K\in L^1_{\rm loc}(\R_+,\R)$ and $G_0$ locally bounded. Assume that there exists a  non-decreasing nonnegative continuous and adapted process $\tilde X$  satisfying 
	\begin{align*}
	\tilde X_t = G_0(t)+ \int_0^t K(t-s)\tilde Z_s ds,
	\end{align*}
	where $\tilde Z$ is a semimartingale with characteristics $(\tilde B_t(\omega),\tilde C_t(\omega),\tilde\nu_t(\omega,d\zeta))$ such that
	\begin{align}\label{eq:condgrowthuni}
	| \tilde B_t(\omega) | + |\tilde C_t(\omega) | + \int_{[0,t]\times\R}\zeta^2  \tilde\nu_s(\omega,d\zeta)  \le \kappa_L|\tilde X_t(\omega)|,  \quad t\geq 0, \quad \mbox{for a.e. } \omega,
	\end{align}
	for some constant $\kappa_L$.
	Then,  for all $T >0$,
   \begin{align}\label{eq:apriori1}
	\EE\left[ \sup_{t \leq T}  | \tilde X_t|^2 \right] \leq C(T,\kappa_L) \left(1+\sup_{t \leq T} |G_0(t)|^2+\|K\|_{L^1([0,T])}^2  \right)\left( 1 + \|R\|_{L^1([0,T])} \right),
	\end{align}
	where $C(T,\kappa_L)>0$ depends exclusively on $(T,\kappa_L)$,  and $R$ is the resolvent of the second kind of $C(T,\kappa_L)  \|K\|_{L^1([0,T])} |K|$.
\end{lemma}

\begin{proof}
	Since $\tilde X$ is non-decreasing, we have $\E\left[\sup_{t\leq T} \tilde X^2_t\right] \leq \E[\tilde X_T^2]$. It is therefore enough to prove the  bound \eqref{eq:apriori1} for  $\E[\tilde X_T^2]$. For this, fix $n \geq 1$  and define $\tau_n=\inf\{t \geq 0: |\tilde X_t| \geq n \} \wedge T$. Since $\tilde X$ is adapted with continuous sample paths, $\tau_n$ is a stopping time such that $\tau_n \to T$ almost surely as $n\to \infty$. First observe that 
	\begin{align*}
	|\tilde X_t|\mathbf 1_{\{ t < \tau_n \}}\leq |G_0(t)|\mathbf + \left| \int_0^t K(t-s) \tilde Z_s \mathbf 1_{\{ s < \tau_n \}} ds \right|.
	\end{align*}
	and set $\tilde X^n_t = \tilde X_t\mathbf 1_{\{ t \leq \tau_n \}}$ . An  applications of Jensen's inequality on the normalized measure $|K(t)|dt/\|K\|_{L^1([0,T])}$ yields
	\begin{align*}
	|\tilde X^n_t|^2 &\leq  2\left| G_0(t)\right|^2 + 2 \left|\int_0^t K(t-s) \tilde Z_s \mathbf 1_{\{ s < \tau_n \}} ds\right|^2   \\
	&\leq  2 \sup_{r \leq T} \left| G_0(r)\right|^2 +  2\|K\|_{L^1([0,T])}\int_0^t |K(t-s)| \left|\tilde Z_s\mathbf 1_{\{ s < \tau_n \}}\right|^2  ds.
	\end{align*}
 $\tilde Z$ admits the decomposition $\tilde Z=\tilde B + \tilde M^c + \tilde M^d$ such that $\E[|\tilde M^c_t|^2]=\E[\tilde C_t]$ and $\E[|\tilde M^d_t|^2]=\E[\int_{[0,t]\times \R} \zeta^2 \tilde \nu_s(\cdot,d\zeta)]$, so that Jensen's inequality combined with the bound \eqref{eq:condgrowthuni} yield 
	\begin{align}\label{eq:Znbound} 
	\E[\tilde Z^2_s \mathbf 1_{\{ s < \tau_n \}}] \leq 3(\kappa_L^2 + \kappa_L)(1+\E[|\tilde X^n_s|^2]), \quad s\geq 0. 
	\end{align}
	Combining the above, we get for a constant $C$ depending exclusively on $(T,\kappa_L)$ that may vary from line to line: 
	\begin{align*}
	\EE\left[|\tilde X^n_t|^2\right] 
	&\leq  C\sup_{r \leq T}\left| G_0(r)\right|^2 + C  \|K\|_{L^1([0,T])}\int_0^t |K(t-s)| \left(1+\EE\left[ |\tilde X^n_s|^2 \right] \right)  ds\\
	&\leq C \left ( 1+ \sup_{r\leq T}\left| G_0(r)\right|^2 + \|K\|_{L^1([0,T])}^{2} \right) \left( 1+  \|R\|_{L^1([0,T])}\right).
	\end{align*}
	where the last line follows from  the generalised Gronwall inequality for convolution equations with $R$  the resolvent of $C\|K\|_{L^1([0,T])}^{p-1} |K|$, see \citet[Theorem 9.8.2]{gripenberg1990volterra}. The claimed estimate \eqref{eq:apriori1} now follows by sending $n\to \infty$ and using Fatou's Lemma in the above. 
\end{proof}

\begin{remark}\label{R:momentbound} If $Z$ is a semimartingale with characteristics \eqref{eq:charZ}, the growth condition \eqref{eq:condgrowthuni} is clearly satisfied with 
	$\kappa_L=\left ( b + c + \int_{\R_+} \zeta^2 \nu(d\zeta) \right),$
	recall the assumption $\int_{\R_+} \zeta^2 \nu(d\zeta)<\infty$. Whence, any solution to \eqref{eq:generalsve} satisfies the estimate \eqref{eq:apriori1}. 
\end{remark}

The following lemma establishes an estimate for the modulus of continuity of the process $\bar X= X-G_0$ defined by 
\begin{align}\label{eq:module}
w_{\bar X,T}(\delta)=\sup\{ |\bar X_s -\bar X_t|: s,t\leq T \mbox{ and } |s-t|\leq \delta \}, \quad 0<\delta \leq 1.
\end{align}  

\begin{lemma}\label{L:module}
Fix $K \in L^1_{\rm loc}(\R_+,\R) $ non-increasing and $G_0$ a locally bounded function.  Let $X$ denote a solution to \eqref{eq:generalsve} for the input $(G_0,K,b,c,\nu)$ and set $\bar X = X-G_0$. Then, for any $T>0$ and $\delta\leq 1$,
 \begin{align*}
 \E\left[ w_{\bar X,T}(\delta) \right] \leq  3(\kappa_L^2 + \kappa_L)\left(1+\E\big[X_T^2\big]\right) \left(\int_0^{\delta} |K(s)|ds + \int_0^T  (K(s)-K(s+\delta))ds  \right) 
 \end{align*}
 with $\kappa_L = \left ( b + c + \int_{\R_+} \zeta^2 \nu(d\zeta) \right)$.
\end{lemma}

\begin{proof}
	Fix $s,t \leq T$ such that $|t-s|\leq \delta$. We first write
	\begin{align*}
	\bar X_s - \bar X_t = \int_{s\wedge t}^{s\vee t} K(s\vee t-u)Z_u du +  \int_0^{s\wedge t} (K(s\vee t-u)-K(s\wedge t-u))Z_u du. 
	\end{align*} 
	Whence, 
	\begin{align*}
|\bar X_s - \bar X_t| &\leq  \sup_{u\leq T}|Z_u| \left( \int_{s\wedge t}^{s\vee t} |K(s\vee t-u)|du+ \int_0^{s\wedge t} |K(s\vee t-u)-K(s\wedge t-u)| du   \right)\\
&\leq \sup_{u\leq T}|Z_u| \left(\int_0^{\delta}|K(u)|du + \int_0^{s\wedge t} |K(s\vee t -s\wedge t +u)-K(u)|du\right) \\
&\leq \sup_{u\leq T}|Z_u| \left(\int_0^{\delta}|K(u)|du + \int_0^{T} \left(K(u)-K(u+\delta)\right)du\right)
 \end{align*} 
	where the last inequality follows from the fact that $K$ is non-increasing and $(s\vee t  -s \wedge t) \leq \delta$.  The claimed estimate follows upon taking the supremum over $s,t$ and the expectation, using the Burkholder-Davis-Gundy inequality for the local martingale parts of $Z$ and Remark~\ref{R:momentbound}. 
\end{proof}

\section{Tightness and stability}\label{S:stability}
In this section, we prove our general tightness and stability result:  Theorem~\ref{T:stability}.   One can appreciate the formulation \eqref{eq:generalsve} and the affine structure of the characteristics \eqref{eq:charZ} for the stability argument.  


We start with a preliminary lemma. 

\begin{lemma}\label{L:lemmaKn}
		Let  $(K^n)_{n \geq 1}$ be a   sequence of locally integrable kernels $K^n$ such that  $   \int_0^t| K^n(s)-K(s)|ds \to 0$, $\mbox{as $n\to \infty$}$, $t \geq 0$, for some kernel $K \in L^1_{\rm loc}(\R_+,\R)$. Then, for all $T>0$,
		\begin{align}\label{eq:limspuKn}
		\lim_{\delta \to 0} \limsup_{n \to \infty}\left( \int_0^{\delta} |K^n(s)| ds + \int_0^T |K^n(s+\delta)-K^n(s)|ds \right)=0.
		\end{align}
\end{lemma}

\begin{proof}
	Fix $\varepsilon>0$ and $T>0$. Since   $K \in L^1_{\rm loc}(\R_+,\R)$, it is $L^1$--continuous, see \citet[Lemma 4.3]{brezis2010functional}, so that we can fix $\delta<1$ such that 
	\begin{align*}
	\int_0^{\delta} |K(s)| ds + \int_0^T |K(s+\delta)-K(s)|ds  \leq \frac{\varepsilon} 5. 
	\end{align*}
	Due to the $L^1$-convergence of the kernels $(K^n)_{n\geq 1}$, let $n_{\delta}$ be such that 
	\begin{align*}
	 \int_0^{T+1} |K^n(s)-K(s)| ds  \leq \frac{\varepsilon} 5, \quad n \geq n_{\delta}. 
	\end{align*}
Fixing $n\geq n_{\delta}$ and using the above leads to 
\begin{align}
 \int_0^{\delta} |K^n(s)| ds \leq \int_0^{T+1} |K^n(s)-K(s)|ds +  \int_0^{\delta}  |K(s)|ds \leq \frac{2}5 \varepsilon
\end{align}
and 
\begin{align}
\int_0^{T} |K^n(s+\delta)-K^n(s)| ds &\leq 2\int_0^{T+1} |K^n(s)-K(s)| ds \\
&\quad  \quad + \int_0^{T} |K(s+\delta)-K(s)| ds \\
&\leq \frac{3}5 \varepsilon.
\end{align}
Whence
$$  \int_0^{\delta} |K^n(s)| ds +  \int_0^{T} |K^n(s+\delta)-K^n(s)| ds  \leq \varepsilon,$$
which yields \eqref{eq:limspuKn}. 
\end{proof}

\begin{proof}[Proof of Theorem~\ref{T:stability}]
	Let  $(X^n)_{n \geq 1}$ be a   sequence of continuous non-decreasing solutions   to \eqref{eq:generalsve}, for the respective inputs $(G_0^n,K^n,b^n,c^n,\nu^n)$, that is, for each $n\geq 1$,
	\begin{align}\label{eq:X^n_}
	X^n_t = G^n_0(t) + \int_0^t K^n(t-s)Z^n_s ds, \quad t\geq 0,
	\end{align}
	where $Z^n$ is a semimartingale with characteristics $\left( b^n X^n,c^nX^n,\nu^n(d\zeta)X^n \right)$, defined on some filtered probability space $(\Omega^n,\mathcal F^n ,(\mathcal F^n_t)_{t\geq 0} ,\P^n)$.
	
$\bullet$ Fix $T>0$. We  argue \textit{tightness} of $(X^n,Z^n)$ on the space $C([0,T],\R) \times D([0,T],\R)$, where $D([0,T],\R)$ is the Skorokhod space endowed with the $J_1$ topology. To prove tightness of $(X^n)_{n\geq 1}$, we start by observing that due to the  uniform convergence of $(G_0^n)_{n \geq 1}$ in \ref{eq:stabilityiv}, it suffices to obtain the tighthness of the sequence  $\bar X^n= X^n-G_0^n$. To this end, we make use of the probabilistic counterpart of the Arzéla-Ascoli theorem given in \citet[Theorem 7.2]{Billingsley_1999}.    By  Chebyshev's inequality, it suffices to prove that 
\begin{align}\label{eq:tightness}
\sup_{n \geq 1} \E[\sup_{t\leq T}\bar X_t^n] < \infty  \quad \mbox{and} \quad \lim_{\delta \to 0} \limsup_{n\geq \infty} \E[w_{\bar X^n,T}(\delta)]= 0 
\end{align}
where $w$ is the modulus of continuity defined as in \eqref{eq:module}.
   By virtue of the continuous dependence of the resolvent on the kernel in $L^1$,   the $L^1$--convergence of $K^n$ in \ref{eq:stabilityii} implies the $L^1$--convergence of the respective sequence of resolvents $(R^n)_{n \geq 1}$, see \citet[Theorem 2.3.1]{gripenberg1990volterra}. Thus, the sequences $(\|K^n\|_{L^1([0,T])})_{n \geq 1}$ and  $(\|R^n\|_{L^1([0,T])})_{n \geq 1}$ are uniformly bounded in $n$.  Furthermore, it follows from \ref{eq:stabilityi} that $(b^n,c^n,\int_{\R}\zeta^2 \nu^n(d\zeta))_{n\geq 1}$ are uniformly bounded in $n$  so that the coefficient $\kappa_L$ appearing in \eqref{eq:condgrowthuni} for $X^n$, recall Remark~\ref{R:momentbound}, does not depend on $n$. Therefore, recalling \ref{eq:stabilityiv}, $\sup_{t\leq T} G_0^n(t)$ is uniformly bounded in $n$ and  the
bound in \eqref{eq:apriori1} for each $X^n$ does not depend on $n$, yielding   $ \sup_{n \geq 1} \E\left [ \sup_{t\leq T} |\bar X^n_t|^2\right]<\infty$. From there, an application of Lemmas~\ref{L:module} and \ref{L:lemmaKn} lead to \eqref{eq:tightness} and the tightness of $(X^n)_{n\geq 1}$ on $C([0,T],\R)$ follows. We claim that the sequence $(Z^n)_{n\in\N}$ is tight on $D([0,T],\R)$. To prove this we verify the conditions in  \citet[Theorem~VI.4.18]{Jacod_2003}. We first note that for any $a>0$, $\varepsilon>0$,  $N\leq T$, we have
\begin{align*}
\P\left( \int_{[0,N]\times\R} \bm1_{|\zeta|>a} dX^n_s\nu^n(d\zeta)  > \varepsilon \right) &\le \frac{1}{a^2\varepsilon} \sup_{n\geq 1}\int_{\R} |\zeta|^2 \nu^n(d\zeta) \sup_{m\geq 1} \E\left[ X^{m}_T  \right] 
\end{align*}
Therefore,
\[
\lim_{a\to\infty}\sup_{n\in\N} \P\left( \int_{[0,N]\times \R} \bm1_{|\zeta|>a} dX^n_s \nu^n(d\zeta) > \varepsilon \right) = 0.
\]
Furthermore, since $(X^n)_{n\geq 1}$ is tight on $C([0,T],\R)$, {the first two `modified' characteristics $(b^nX^n,(c^n+\int_{\R}\zeta^2 \nu^n(d\zeta) )X^n)$ of $Z^n$ are $C$-tight by virtue of \ref{eq:stabilityi}.}
 In addition, for $p\in \N$ and $h_p(\zeta)=(p|\zeta|-1)^+ \wedge 1$,  $X^n \int_{\R} h_p(\zeta)\nu^n(d\zeta)$ is also $C$-tight thanks to  \ref{eq:stabilityi}. Whence, we may  apply \citet[Theorem~VI.4.18]{Jacod_2003} to get that $(Z^n)_{n\in\N}$ is tight on $D([0,T],\R)$. Finally, by passing to a further subsequence, we have $(X^n,Z^n)\Rightarrow(X,Z)$ on $C([0,T],\R)\times D([0,T],\R)$ for some limiting process $(X,Z)$ defined on a probability space $(\Omega,\mathcal F, \P)$.

$\bullet$  We now move to the \textit{stability} part. We start by proving that $Z$ is a semimartingale with characteristics $(bX,cX,\nu(d\zeta)X)$ with respect to the filtration  $\F=(\mathcal F_{t})_{t\leq T}$ generated by $(X,Z)$. Since $X^n$ and $X$ are continuous, we have $X^n\Rightarrow X$ on $D([0,T],\R)$ so that,   by virtue of   \ref{eq:stabilityi}, we have, by passing to a further subsequence,
$$ \left(X^n,Z^n, b^nX^n,\left(c^n+\int_{\R}\zeta^2 \nu^n(d\zeta) \right)X^n \right) \Rightarrow  \left(X,Z, bX,\left(c+\int_{\R}\zeta^2 \nu(d\zeta) \right)X \right) $$
and for any  continuous and bounded function $h$ vanishing around zero:
$$ \left(X^n,Z^n,  \int_{\R_+}h(\zeta) \nu^n(d\zeta) X^n\right) \Rightarrow  \left(X,Z, \int_{\R_+}h(\zeta) \nu(d\zeta) X\right).$$
 It follows from \citet[Theorem IX-2.4 and Remark IX-2.21]{Jacod_2003} that $Z$ is a semimartingale with  characteristics $(bX,cX,\nu(d\zeta)X)$ with respect to the filtration  $\F$.\\
An application of Skorokhod's representation theorem provides the existence of a common filtered probability space $(\Omega,\mathcal F ,(\mathcal F_t)_{t\geq 0} ,\P)$ supporting a sequence of copies $(X^n,Z^n)_{n\geq 1}$ that converges on $C([0,T]\times \R)\times D([0,T]\times \R)$, almost surely,  along a subsequence, towards a copy of $(X,Z)$. Keeping the same notations for these copies, we have 
\begin{align}\label{eq:skoconv}
\|X^n - X\|_{C([0,T],\R)} \to 0, \quad  \|Z^n - Z\|_{D([0,T],\R)} \to 0,   \quad  \P-a.s., \quad  \mbox{as } n \to \infty.
\end{align}
Now  fix $t\leq T$ and write 
\begin{align*}
\int_0^t K^n(t-s)Z^n_s ds - \int_0^t K(t-s)Z_s ds&=  \int_0^t (K^n(t-s)-K(t-s))Z^n_s ds \\
&\quad + \int_0^t K(t-s)(Z^n_s-Z_s) ds\\
&=\bold{I}_n+\bold{II}_n.
\end{align*}
Due to the convergence of $(Z^n)_{n\geq 1}$ on $D([0,T],\R)$,  $Z^n_s \to Z_s$ $\P\times dt$--almost everywhere and $\sup_{n\geq 1} \sup_{s\leq T} |Z^n_s|<\infty$ so that $\bold{I}_n\to 0$ as $n\to \infty$ by virtue of \ref{eq:stabilityii}  and  $\bold{II}_n\to 0$ by dominated convergence.   This shows that $\int_0^t K^n(t-s)Z^n_s ds \to \int_0^t K(t-s)Z_s ds$.  Combined with \ref{eq:stabilityiv}, we get, after taking the limit in \eqref{eq:X^n_}, that 
\begin{align*}
X_t = \lim_{n \to \infty} X^n_t = G_0(t) + \int_0^t K(t-s)Z_s ds,
\end{align*}
for all $t\leq T$. Since $X,G_0$ and $t\mapsto \int_0^t K(t-s)Z_s ds$ are continuous,  one can interchange the quantifiers so that the previous identity holds for all $t\leq T$, $\P$ almost surely.  Finally,  each $X^n$ being non-decreasing and nonnegative, the limit process $X$ is again non-decreasing and nonnegative, which ends the proof. 
\end{proof}

\section{Existence for $L^1$-kernels}\label{S:existence}

\subsection{Existence for the stochastic Volterra equation}\label{S:existencestochastic}
In this section we  prove the existence of solutions for the stochastic Volterra equation \eqref{eq:generalsve}, i.e. Theorem~\ref{T:MainIntro}. We proceed in two steps. We first prove the claimed existence for smooth kernels $K\in C^1$ and finite measures $\nu$. Second, we apply a density argument, i.e.~Theorem \ref{T:stability}, to obtain the existence for $K \in L^1_{\rm loc}$ with possibly infinite activity jumps. We point out that for $L^2$-kernels and possibly infinite activity jumps, existence was already obtained by \citet{cuchiero2020generalized} using infinite dimensional Markovian lifts. However, the set of admissible input curves there is different than $\mathcal G_K$, recall \eqref{eq:notsofriendlyset}, and the assumptions on $K$ are  different. For this  reason, we provide another proof in the $L^2$-setting by working directly on the level of the scalar stochastic Volterra equation, in the spirit of \citet{AJEE:19markovian,Abi_Jaber_2017}.

To this end, for a stopping time $\tau$ we extend the definition of the set $\mathcal G_K$  in \eqref{eq:notsofriendlyset2} by considering 
\begin{align}
\!\!\!\mathcal G^{\tau}_{K} \!=\! \big\{ \!(g(s))_{s\geq 0} &\!\mbox{ adapted process:  satisfying }\! \eqref{eq:notsofriendlycondition} \nonumber \\
&\quad  \!\mbox{ and continuous on }\! [\tau(\omega),\infty) \mbox{ with }\! g(\tau(\omega))\geq 0  \mbox{ a.s.}\big\}. \label{eq:notsofriendlyset2}
\end{align} 

 The following lemma provides some elementary results on $\mathcal G_K$.
\begin{lemma}\label{L:Gtau} Let $K$ be nonnegative, non-increasing and continuous on $[0,\infty)$ and $\tau$ a stopping time. 
	\begin{enumerate}
		\item  \label{L:Gtaui}
		Let $\eta$ be a nonnegative random variable, then  $s\mapsto \bm 1_{\tau \leq s} K(s-\tau)\eta$ belongs to $\mathcal G^{\tau}_{K}$, 
		\item \label{L:Gtauii}
		If $g_1,g_2  \in \mathcal G^{\tau}_K$, then $g_1 + g_2 \in \mathcal G^{\tau}_{K}$. If $\tau\leq \tau'$, then $\mathcal G_K^{\tau} \subset \mathcal G_K^{\tau'}$.
	\end{enumerate}
\end{lemma}

\begin{proof}
	\ref{L:Gtauii} is straightforward from  the affine structure of \eqref{eq:notsofriendlycondition}.
	We  prove \ref{L:Gtaui}. Clearly  $g:s\mapsto \bm 1_{\tau \leq s} K(s-\tau)\eta$ is continuous on $\{\tau(\omega)\leq t\}$  such that  $g(\tau)=K(0)\eta \geq 0$ a.s. To argue \eqref{eq:notsofriendlycondition}, we fix $h\geq 0$.  It follows from 
 \citet[Lemma B.2 and Remark B.3]{AJEE:19markovian} that 
$$ \Delta_h K = (\Delta_h K *L) (0) K + d(\Delta_h K *L)*K.$$
 Whence, on $\{\tau \leq t\}$:
 \begin{align}
 \Delta_h g(t)&= \Delta_h K(t-\tau ) \eta\\
 &= (\Delta_h K *L) (0) K(t-\tau)\eta + (d(\Delta_h K *L)*K)(t-\tau) \eta \\
 &=(\Delta_h K *L) (0) g(t) + \int_0^{t} d(\Delta_h K *L)(ds)\bm 1_{\tau \leq t-s}K(t-s-\tau) \eta\\
 &=(\Delta_h K *L) (0) g(t) + (d(\Delta_h K *L)*g)(t)
 \end{align}
 which yields that the left hand side of \eqref{eq:notsofriendlycondition} is zero, leading to \ref{L:Gtaui}.  
\end{proof}

We  recast the existence  results of  \citet{Abi_Jaber_2017,AJEE:19markovian} obtained for $L^2$-kernels and deterministic input curves $g_0 \in \mathcal G_K$ in the absence of jumps in our framework  to allow for random input curves.

\begin{lemma}\label{L:existencesvecontinuous}
Fix $\beta,\sigma \in \R$  and let $K \in C^1[0,\infty)$ satisfying \eqref{eq:K orthant}. Let $(\Omega,\Fc,(\Fc_t)_{t\leq T}, \P)$ denote a probability space supporting a  Brownian motion $W$. Fix a stopping time $\tau$ and a  process  $\tilde g_{\tau} \in  \mathcal G^{\tau}_{K}$. Then, the equation
		\begin{align}\label{eq:Ytau}
	Y_t = \tilde g_{\tau}(t) + \int_{\tau}^t K(t-s)\beta Y_s ds + \int_{\tau}^t K(t-s)\sigma \sqrt{Y_s} dW_s. 
	\end{align}
	admits a unique nonnegative continuous and adapted strong solution $Y$ on $[\tau,\infty)$.
	Furthermore, $\mathcal G^{\tau}_{K}$ is invariant for the process 
	 \begin{align}
	 \tilde g_t(s)= \tilde g_{\tau}(s) + \int_{\tau}^t K(s-u)\beta Y_u ds + \int_{\tau}^t K(s-u)\sigma \sqrt{Y_u} dW_u, \quad \tau \leq t\leq s,
	 \end{align}
	 meaning that $g_t$ is $\mathcal G^{\tau}_{K}$-valued on $[\![\tau,\infty)\!)=\{(\omega,t): \tau(\omega)\leq t \}$.
\end{lemma}

	\begin{proof}
		We first argue for $\tau\equiv 0$ and deterministic input $g_0\in \mathcal G_K$. The existence of a $\RR_+$--valued  continuous nonnegative weak solution $Y$  to \eqref{eq:Ytau}  follows from
		\citet[Theorem A.2]{AJEE:19markovian}\footnote{We note that all the assumptions   are met there, except  for the local H\"older continuity of $g_0$. This assumption  is only needed to get H\"older sample paths of $X$, which we do not require here. Assumption $(H_0)$  there is satisfied with $\gamma=2$ since $K$ is $C^1$.}. The strong uniqueness of $Y$ follows from \citet[Proposition B.3]{Abi_Jaber_2019}. This yields the strong existence and uniqueness for \eqref{eq:Ytau}. Finally, an application of the second part of 	\citet[Theorem 3.1]{AJEE:19markovian} yields the invariance of $\mathcal G_K$ with respect to $t\mapsto \tilde g_t$, after noticing that 
		$$\tilde g_t(s)=\E\left[ Y_s -\int_t^s K(s-u)\beta Y_u du \Mid \mathcal F_t \right],\quad t\leq s.$$ For arbitrary $\tau$ and random input $\tilde g_{\tau} \in \mathcal G_K^{\tau}$, the result follows by a straightforward  adaptation of the aforementioned results. 
	\end{proof}

	We now construct a solution to \eqref{eq:generalsve}  when $\nu$ is finite and $K$ is continuously differentiable by pasting continuous solutions $Y^i$ to \eqref{eq:Ytau} on each interval $[\tau_i,\tau_{i+1})$ between two consecutive jumps. 
	
\begin{lemma}\label{T:L2existence}
	Let $b\in \R,c \geq 0$ and $\nu$ be a nonnegative finite measure supported on $\R_+$ such that $\int_{\R_+} \zeta^2 \nu(d\zeta)<\infty$. Fix $K \in C^1[0,\infty)$ satisfying \eqref{eq:K orthant}	and let $g_0 \in \mathcal G_K$. There exists a non-decreasing nonnegative continuous solution $X$ to \eqref{eq:generalsve} for the input $G_0(t)=\int_0^t g_0(s)ds$.
\end{lemma}

\begin{proof}
	Using Lemma~\ref{L:spotvariance}, it is enough to first prove the existence of a càdlàg nonnegative solution 
	$Y$ to the equation 
	\begin{align}\label{eq:Yfinitejumps}
	Y_t = g_{0}(t) + \int_{0}^t K(t-s)dZ_s. 
	\end{align}
	where $Z$ is a semimartingale with differential characteristics with respect to the Lebesgue measure $(bY , cY, \nu(d\zeta)Y)$, and then set $X=\int_{0}^{\cdot} Y_s ds$ to obtain the desired solution to \eqref{eq:generalsve}. Since $\nu$ is finite,    finding a solution to equation \eqref{eq:Yfinitejumps} is equivalent to solving 
	\begin{align}\label{eq:Yfinitejumps2}
	Y_t = g_{0}(t) +  \sum_{i\geq 0}  \bm 1_{\tau_i\leq t}   \int_{\tau_i}^{t\wedge \tau_{i+1}}& K(t-s)\left( \beta Y_s ds +   \sqrt{c Y_s} dW_s\right) \nonumber   \\
	&\quad\quad\quad \quad\quad + \sum_{i\geq 1}  \bm 1_{\tau_i\leq t}   K(t-\tau_i)J_i,
	\end{align} 
		where $\beta = b-\int_{\R_+} \zeta \nu(d\zeta)$, $J_i$ are the jump sizes of $Z$  distributed  according to $\nu(d\zeta)/\nu(\R_+)$ and arriving at the jump times $\tau_i$ with instantaneous intensity $Y_t\nu(\R_+)$ and  $W$ is a Brownian motion on some  filtered probability space  $(\Omega,\mathcal F, (\mathcal F_t)_{t\geq 0},\P )$.\\
	Our strategy for constructing a solution $Y$ to the above equation consists in pasting continuous solutions $Y^i$ on each interval $[\tau_i,\tau_{i+1})$ between two consecutive jumps. More precisely, fix a filtered probability space  $(\Omega,\mathcal F, (\mathcal F_t)_{t\geq 0},\P )$ supporting a Brownian motion $W$ and a sequence of independent random variables $(E_i,U_i)_{i\geq 1}$ with $E_i$ exponentially distributed with intensity $1$  and $U_i$ standard uniform, and let $F_\nu$ denote the cumulative distribution function with density $\nu(d\zeta)/\nu(\R_+)$. We set $\tau_0=0$ and  we assume that for each $i\geq 0$ we have a unique nonnegative continuous solution $Y^i$ on $[\tau_i, \tau_{i+1})$ for the following system of inductive equations
\begin{align}
Y^i_t &= \tilde g^i(t) + \int_{\tau_i}^t K(t-s)\left( \beta Y^i_s ds  +  \sqrt{c Y^i_s}dW_s \right) \label{eq:Yi}\\
\tilde g^i(t)&:= g_0(t) + \sum_{j=0}^{i-1} \int_{\tau_{j}}^{\tau_{j+1}} K(t-s)\left( \beta Y_s^{j}ds +  \sqrt{c Y_s^{j}}dW_s  \right) + \sum_{j=1}^i K(t-\tau_j)J_j  \label{eq:gi}
\end{align}
with the convention that $\sum_{i=0}^{-1} = \sum_{j=1}^0=0$,	
$\tau_{i+1}= \tau_{i} + \delta_i$  with $\delta_i = \inf\{ s >0: \nu(\R_+) \int_{\tau_{i}}^{\tau_{i}+s} Y^{i}_u du \geq E_i   \}$ (with the convention that $\inf \emptyset=\infty$) and  $J_i=F_{\mu}^-(U_i)$.  Then, by a localization argument $\tau_i \to \infty$ and it is straightforward to check that the process $Y$ defined by 
$$  Y_t = \sum_{i\geq 0} Y^i_t \bm 1_{\tau_i\leq t < \tau_{i+1}}$$
is the unique càdlàg continuous nonnegative solution to \eqref{eq:Yfinitejumps2}. 
By the first part of Lemma~\ref{L:existencesvecontinuous}, the existence and uniqueness of a solution $Y^i$ is ensured provided that  the process $\tilde g^i$ is $\Gcal_K^{\tau_i}$-valued, for each $i\geq 1$. We now prove this claim by induction using the second part of Lemma~\ref{L:existencesvecontinuous}.
\textit{Initialization:} for $i=0$, $\tilde g^0=g_0$ is deterministic and $\Gcal_K$-valued by assumption. \textit{Heredity:} fix $i\geq 0$ and assume that $\tilde g^i \in \Gcal_K^{\tau_i}$-valued. Fix $Y^i$ the unique nonnegative solution to \eqref{eq:Yi} on $[\tau_i,\infty)$ obtained from Lemma~\ref{L:existencesvecontinuous}  for the input $\tilde g^i$. The second part of Lemma~\ref{L:existencesvecontinuous} yields that the process 
	 \begin{align}
\tilde g^i_t= \tilde g^i + \int_{\tau_i}^t K(\cdot-u)\beta Y^i_u ds + \int_{\tau_i}^t K(\cdot-u)\sqrt{c Y^i_u} dW_u, 
\end{align}
is $\Gcal_K^{\tau_i}$-valued on $[\![\tau_i,\infty)\!)$. In particular, the stopped process
$ \tilde g^i_{\tau_{i+1}}$ belongs to $\Gcal_K^{\tau_i}$. 
We now  observe that \eqref{eq:gi} can be re-written in terms of $\tilde g_{\tau_{i+1}}^i$:
\begin{align}
 \tilde g^{i+1}(t)&= \tilde g^i(t) + \int_{\tau_i}^{\tau_{i+1}} K(t-s)\left( \beta Y^i_s ds + \sqrt{ c Y^i_s}dW_s \right) + K(t-\tau_{i+1})J_{i+1}\\
 &= \tilde g^i_{\tau_{i+1}}(t) + K(t-\tau_{i+1})J_{i+1},
\end{align}
on $\{t\geq \tau_{i+1}\}$.
Since $K$ is nonnegative and $J_{i+1}\geq 0$, recall that $\nu$ is supported on $\R_+$, an application of Lemma~\ref{L:Gtau} yields that   $\tilde g^{i+1} \in \mathcal G_K^{\tau_{i+1}}$. This proves the induction and ends the proof.
\end{proof}

For the general case, we use a density argument, i.e.~Theorem~\ref{T:stability}, to obtain the existence statement in Theorem~\ref{T:MainIntro}.

\begin{proof}[Proof of the existence statement in Theorem~\ref{T:MainIntro}]
	Fix $n \geq 1$. 
	Let $\nu^n(d\zeta)=\bm1_{\zeta \geq \frac  1 n} \nu(d\zeta)$ and $K^n=\Delta_{1/n}K$. Then, $\nu^n$ is a nonnegative  finite measure  supported on $\R_+$  such that  $\int_{\R_+} \zeta^2 \nu^n(d\zeta)<\infty$.
	 An application of Lemma~\ref{T:L2existence} yields the existence of  a non-decreasing and continuous process $X^n$   solution to \eqref{eq:generalsve} with the inputs $(G_0^n,K^n,b,c,\nu^n)$, where $G_0^n=\int_0^{\cdot} g_0^n(s)ds$. Each $X^n$ being non-decreasing, the claimed existence now follows  from Theorem~\ref{T:stability}, once we prove that  conditions \ref{eq:stabilityi}-\ref{eq:stabilityiv} are satisfied. \ref{eq:stabilityi} is clearly satisfied. \ref{eq:stabilityii} holds  by the $L^1$--continuity of the kernel $K$, see \citet[Lemma 4.3]{brezis2010functional}.   Finally, to obtain \ref{eq:stabilityiv} we first observe that $g_0^n$  is nonnegative, this follows from \eqref{eq:notsofriendlycondition} evaluated at  $t=0$. Whence, $G^n_0$ is non-decreasing with a continuous pointwise limit $G_0$. An application of Dini's second theorem yields \ref{eq:stabilityiv}.
	The proof is complete.
\end{proof}

\subsection{Existence for the Riccati--Volterra equation}\label{S:existenceRiccati}
In this section we prove Theorem~\ref{T:existenceRiccatiL1}.

\begin{proof}[Proof of Theorem~\ref{T:existenceRiccatiL1}]
	 We first note that the function $z\mapsto \int_{\R_+} \left(e^{z\zeta}-1 -z\zeta \right)\nu(d\zeta)$ is continuous on $\mathcal U= \{z \in \mathbb C: \Re(z)\leq 0 \}$. Whence, since $f_2$ is continuous and $\Re(f_2) \leq 0$ by \eqref{eq:coeffaf}, we obtain that $(s,u)\mapsto \int_{\R_+} \left(e^{(f_2(s)+u)\zeta}-1 -(f_2(s)+u)\zeta \right)\nu(d\zeta)$ is continuous on $[0,T]\times \mathcal U$. We define  
	$$  \tilde F(s,u)= F(s,\Re(u)\bm 1_{\Re(u)\leq 0} + i \Im(u)),$$
	where we recall that $F$ is given by \eqref{eq:RiccatiL2b}. Then, $\tilde F$ is continuous on $[0,T] \times \mathbb C$ so that an application of \citet[Theorem 12.1.1 and the comment below]{gripenberg1990volterra} (on the positive and negative parts) yields the local existence of a continuous solution $\psi$ to the equation 
	\begin{align}\label{eq:ricFtilde}
	\psi(t)=\int_0^t K(t-s)\tilde F(s,\psi(s))ds
	\end{align}
	on the interval $[0,T_{\infty})$ where $T_{\infty}=\inf\{t: |\psi(t)|=\infty\}$. In order to obtain the claimed existence for \eqref{eq:RiccatiL2a}-\eqref{eq:RiccatiL2b} it suffices to prove that $T_{\infty}=+\infty$ and that $\Re(\psi(t))\leq 0$, for all $t\geq 0$. \\
	\textit{Step 1. We first prove that $\Re(\psi)\leq 0$ on $[0,T_{\infty})$.}  Fix $T<T_{\infty}$ and denote by $\tilde \psi_{\rm r}(t)= \Re(\psi(t))\bm 1_{\Re(\psi(t))\leq 0}$. Taking real parts in \eqref{eq:ricFtilde}, we get that $\chi=\Re(\psi)$ satisfies
\begin{align}\label{eq:chiabove1}
\chi(t) =\int_0^t K(t-s)\Re(\tilde F(s,\psi(s))ds, \quad t\leq T,
\end{align}
with 
\begin{align*}
\Re(\tilde F(s,\psi(s))&=\Re(f_0(s)) + \frac 1 2 c \Re(f_1(s))^2 - \frac c 2  ( \Im(f_1(s))  +\Im(\psi(s)))^2 \\
&\quad +\left(b + c\Re(f_1(s))+ \frac c 2 \tilde \psi_{\rm r}(s)\right)\tilde \psi_{\rm r}(s)\\
&\quad  + \int_{\R_+} \left(\cos(\Im(f_2(s)+\psi(s))\zeta)e^{(\Re(f_2(s))+\tilde\psi_{\rm r}(s))\zeta}-1 -(\Re(f_2(s))+\tilde\psi_{\rm r}(s))\zeta \right)\nu(d\zeta).
\end{align*}
	Since $\cos \leq 1$ we obtain, for $v\leq 0$,
	\begin{align*}
	 \cos(a)e^{v\zeta}-1 -v\zeta \leq e^{v\zeta}-1 -v\zeta= \zeta^2\int_v^0 \int_r^0 e^{vl}dl dr \leq  \frac{\zeta^2v^2}{2}.
	\end{align*}
	Whence, 
	\begin{align}\label{eq:chiabove2}
\Re(\tilde F(s,\psi(s)) &\leq  z(s)\tilde \psi_{\rm r}(s)+w(s), \quad s \leq T,
	\end{align}
	with 
			\begin{align*}
	z(s)&=    b + c\Re(f_1(s)) + \frac c 2 \tilde\psi_{\rm r}(s) + \int_{\R_+} \zeta^2 \nu(d\zeta)\left(\Re(f_2(s))+\frac  1 2 \tilde\psi_{\rm r}(s)   \right)\\
	w(s)&= \Re(f_0(s)) + \frac 1 2 c \Re(f_1(s))^2  + \frac 1 2 \int_{\R_+} \zeta^2 \nu(d\zeta)(\Re(f_2(s))^2 - \frac c 2  ( \Im(f_1(s))  +\Im(\psi(s)))^2 .
	\end{align*}
	Observing  that $\tilde \psi_{\rm r}(t)=\bm 1_{\Re(\psi(t))\leq 0}\chi(t)$, recalling that $K$ is nonnegative and combining \eqref{eq:chiabove1}-\eqref{eq:chiabove2} leads to
	\begin{align*}
	\chi(t)\leq \int_0^t K(t-s) (\tilde z(s)\chi(s)+ w(s))ds,
	\end{align*}
	with $\tilde z(s)=\bm 1_{\Re(\psi(s))\leq 0} z(s)$. Whence, 
	\begin{align*}
	\chi(t)= - h(t) + \int_0^t K(t-s)  (\tilde z(s)\chi(s)+ w(s))ds
	\end{align*}
	for some nonnegative function $h$.   Denoting by $\tilde \chi$ the solution to the linear equation 
\begin{align}\label{eq:tildechi}
 \tilde \chi(t) = \int_0^t K(t-s)(\tilde z(s)\tilde \chi(s)+ w(s))ds
\end{align}
 we obtain that $\delta=(\tilde \chi - \chi)$ solves the linear equation 
$$  \delta(t)=h(t)+ \int_0^t K(t-s)\tilde z(s)\delta(s)ds.$$ 
Since, $h \geq 0$, an application of  \citet[Theorem C.1]{Abi_Jaber_2019}\footnote{Inspecting the proof of  \citet[Theorem C.1]{Abi_Jaber_2019} one can see that the $L^2$ integrability on the kernel assumed there does not play any role, the result remains clearly valid  for $K \in L^1([0,T],\R)$. Similarly, the continuity assumption on $z$ there can be weakened to local boundedness.} leads to $\delta\geq 0$ on $[0,T]$ so that  
\begin{align}\label{eq:chichitilde}
\chi(t)\leq \tilde \chi(t), \quad t\leq T.
\end{align}
We now argue that $\tilde \chi \leq 0$. 
	 By virtue of \eqref{eq:coeffaf}, $w(s)\leq 0$ for all $s\leq T$, so that another application  of  \citet[Theorem C.1]{Abi_Jaber_2019} on the equation \eqref{eq:tildechi} leads to $\tilde \chi(t)\leq 0$, for all $t\leq T$. Finally, from \eqref{eq:chichitilde}, we obtain that  $\Re(\psi(t))=\chi(t)\leq \tilde \chi(t)\leq 0$, for all $t\leq T$. The claimed conclusion follows by arbitrariness of $T<T_{\infty}$.\\
	 \textit{Step 2. We now argue that $T_{\infty}=\infty$.}  Fix $T<T_{\infty}$.  By the above we have $\tilde \psi_{\rm r} = \Re(\psi)$ so that $\tilde F(s,\psi(s))=F(s,\psi(s))$ for all $s< T_{\infty}$. Using this fact in \eqref{eq:ricFtilde}, we observe that $\psi$ solves the linear equation 
	 \begin{align*}
	 h(t)&= \int_0^tK(t-s) \left(\frac{c\psi(s)} 2 h(s)+ \alpha(s) \right)ds,
	 \end{align*}
	 with $\alpha$ defined by 
	 $$ \alpha(s)=F(s,\psi(s)) - \frac{c}{2}\psi(s)^2.$$
	 Since $\Re(\psi)\leq 0$, an application of  \citet[Theorem C.4]{Abi_Jaber_2019}\footnote{Again the $L^2$ integrability on the kernel assumed there does not play any role, the result remains clearly valid  for $K \in L^1([0,T],\R)$.} yields that $\sup_{t\leq T} |\psi(t)|<\infty$. By arbitrariness of $T$ we obtain that $T_{\infty}=\infty$. The proof is complete.
	\end{proof}

\section{Weak uniqueness and the Fourier--Laplace transform}\label{S:uniqueness}
In this section, we prove Theorem~\ref{T:MainIntrouniqueness}.

Throughout this section,  we fix  $T\geq 0$, $f_0,f_1,f_2:[0,T]\to \mathbb C$ continuous functions, $G_0$ a non-decreasing continuous function  and $K \in L^1([0,T],\R)$. We let  $\psi \in C([0,T],\C)$  denote a solution to the Riccati equation \eqref{eq:RiccatiL2a}-\eqref{eq:RiccatiL2b} such that \eqref{eq:condmomentpsi} holds and  $X$ be  a non-decreasing nonnegative continuous weak solution to \eqref{eq:generalsve} for the input $(G_0,K,b,c,\nu)$. We recall the decomposition  of $Z$ in \eqref{eq:decompZ} and we define the process $V^T$:
{\begin{align}
V_t^T &= V_0^T + \int_0^t \beta^T_s dX_s + \int_0^t \left( f_1(T-s)+\psi(T-s) \right)dM^c_{s} \nonumber
\\ & \quad \quad \quad \quad \quad \quad \quad \quad \quad \quad \quad \quad + \int_0^t  \left( f_2(T-s)+\psi(T-s) \right) dM^d_s\label{eq:processweak1}\\
\beta^T_s &= -\frac c 2 \left( f_1(T-s)+  \psi(T-s)\right)^2 \nonumber \\
&\quad \quad  - \int_{\R_+} \left(e^{(f_2(T-s) + \psi(T-s))\zeta} -1 - (f_2(T-s) + \psi(T-s))\zeta \right) \nu(d\zeta)   \label{eq:processweakbeta}\\
V_0^T &= \int_0^T F(T-s,\psi(T-s)) dG_0(s) \label{eq:processweak2}.
\end{align}
We note that the integral involving $\nu$ is well-defined by virtue of \eqref{eq:condmomentpsi} combined with the local boundedness of $(\psi,f_2)$ and the inequality $|e^{\alpha \zeta}-1-\alpha \zeta|\leq e^{\Re(\alpha)\zeta}\alpha\zeta^2/2$.} The Lebesgue-Stieltjes integrals are well-defined since $(\psi,f_1,f_2)$ are continuous and $(G_0,X)$ are of locally bounded variation.

{$(V^T)_{t\leq T}$ is a semimartingale 
and a straightforward application of It\^o's Lemma yields that the stochastic exponential  $H=\exp(V^T )$ is a complex local martingale with dynamics 
\begin{align*}
dH_t &= H_{t-}dN_t \\
dN_t&= (f_1(T-t) + \psi(T-t))dM^c_t  +  \int_{\R_+} \left(e^{(f_2(T-t)+\psi(T-t))\zeta}-1\right) \left( \mu^Z (dt,d\zeta) -  \nu(d\zeta) dX_t \right) ,
\end{align*}
meaning that $H=\mathcal E(N)$, where $\mathcal E$ stands for the Doléans--Dade exponential.}
 The following lemma, which extends \citet[Lemma 7.3]{Abi_Jaber_2017}, establishes that $H$ is even a true martingale.

\begin{lemma}\label{L:truem}
	Let $g_1,g_2\in L^\infty(\R_+,\R)$ such that 
			\begin{align}\label{eq:condg2}
	\sup_{s\geq 0} \int_{\R_+} e^{g_2(s) \zeta} \zeta^2 \nu(d\zeta) < \infty 
	\end{align}
	and define 
\begin{align}\label{eq:U}
U_t = \int_0^t g_1(s) dM_s^c + \int_{[0,t]\times \R_+} \left(e^{g_2(s)\zeta}-1\right) \left( \mu^Z (ds,d\zeta) -  \nu(d\zeta) dX_s \right) .
\end{align}	
	Then, the Doléans--Dade exponential $\mathcal E (U)$ is a martingale. Furthermore, $H=\exp(V^T )$ is a martingale on $[0,T]$.
\end{lemma}

\begin{proof} \textit{Part 1. Martingality of $M:=\mathcal E(U)$.}
	We first recall that   
	$$M_t =\mathcal E\left(\int_0^{\cdot} g_1(s)dM^c_s\right)\prod_{0<s\leq t} \left(1+\Delta U_s \right)e^{-\Delta U_s}.$$
	 Since $\Delta U_s=(e^{g_2(s)\Delta Z_{s}}-1)>-1$, for all $s\geq 0$, $M$ is a nonnegative local martingale. Whence, it is a supermartingale by Fatou's lemma, and it suffices to show that $\E[M_T]= 1$ for any $T\in\R_+$. To this end, fix $T>0$ and  define the stopping times $\tau_n=\inf\{t\ge0\colon X_t> n\}\wedge T$.  We first argue that  $M^{\tau_n}=M_{\tau_n \wedge \cdot }$ is a uniformly integrable martingale for each $n$ by verifying the condition in \citet[Theorem IV.3]{lepingle1978integrabilite} with the process $y(\omega,t,\zeta)=1_{t\leq \tau_n(\omega)} (e^{g_2(t)\zeta}-1)$. Using the bound $|\alpha \zeta e^{\alpha \zeta}+1 -e^{\alpha \zeta}|\leq \zeta^2 \alpha^2 e^{\alpha \zeta}/2 $ and the boundedness of $g_1,g_2$,  we get that the quantity 
\begin{align*}
\int_0^T \bm 1_{s\leq \tau_n}g_1(s) cdX_s + \int_{[0,T]\times \R_+} \bm 1_{s\leq \tau_n} \left(g_2(s) \zeta e^{g_2(s) \zeta}+1 -e^{g_2(s) \zeta}\right) \nu(d\zeta) dX_s   
\end{align*}	
is bounded by $ \kappa  n \left( 1 + \sup_{s\geq 0} \int_{\R_+} e^{g_2(s)\zeta}\zeta^2 \nu(d\zeta)  \right)$, for some constant $\kappa>0$. The upper bound is finite due to condition \eqref{eq:condg2}. \citet[Theorem IV.3]{lepingle1978integrabilite} can be applied to get that $M^{\tau_n}$ is a martingale for each $n$. Whence,
	\begin{align}
	1=M^{\tau_n}_0 = \E_{\P}\left[ M^{\tau_n}_T \right]=
 \E_{\P}\left[ M_T \bm 1_{\tau_n\geq T} \right]+ \E_{\P}\left[ M_{\tau_n} \bm 1_{\tau_n< T} \right],
 	\end{align}
 	where we made the dependence of the expectation on $\P$ explicit. 
	Since $ \E_{\P}\left[ M_T \bm 1_{\tau_n\geq T} \right] \to \E_{\P}\left[ M_T  \right]$ as $n\to \infty$, by dominated convergence, in order to get that  $\E_{\P}[M_T]=1$, it suffices to  prove that 
	\begin{align}\label{eq:girsanovproof1}
	\E_{\P}\left[ M_{\tau_n} \bm 1_{\tau_n< T} \right]\to 0, \quad \mbox{as } n\to \infty. 
	\end{align}
To this end, since $M^{\tau_n}$ is a martingale, we may define probability measures $\Q^n$ by
	\[
	\frac{d\Q^n}{d\P} = M^{\tau_n}_{\tau_n}.
	\]
	By Girsanov's theorem, see \citet[Theorem III.3.24]{Jacod_2003} (see also the formulation in \citet[Proposition 2.6]{kallsen2006didactic}), the process  $Z$ is a  semimartingale under $\Q^n$ with characteristics 
$$ \left( B^n, cX,  {\int_0^{\cdot}  \bm1_{\{s\leq \tau_n\}} e^{g_2(s)\zeta} \nu(d\zeta)dX_s}\right)$$
where 
$$ B^n=bX + \int_0^{\cdot} \bm1_{\{s\leq \tau_n\}} g_1(s) c dX_s {+\int_{[0,\cdot]\times \R_+} \zeta \left( e^{g_2(s)\zeta}-1 \right)  \bm1_{\{s\leq \tau_n\}} \nu(d\zeta)dX_s }.$$
Under $\mathbb Q^n$, we still have
	\[
	X_t = G_0(t)  + \int_0^t K(t-s) Z_s ds,
	\]
and we observe that, due to the boundedness of $g_1$, the equality $|e^{\alpha \zeta}-1|\leq \zeta (1+e^{\alpha \zeta})$ and \eqref{eq:condg2}, the characteristics of $Z$ under $\Q^n$ satisfy the growth condition in $X$ as in \eqref{eq:condgrowthuni} for some constant $\kappa_{L}$ independent of $n$. Therefore, an application of  Lemma~\ref{L:apriori} yields the moment bound
	\[
 \E_{\Q^n}[ 	\sup_{t\le T} |X_t|^2 ] \le \eta(\kappa_L,T,K,G_0),
	\] 
where $\eta(\kappa_L,T,K,G_0)$ does not depend on $n$.
We then get by an application of Chebyshev's inequality
	\begin{align*}
	\E_{\P}\left[ M_{\tau_n} \bm 1_{\tau_n\leq T} \right] &= 	 \Q^n(\tau_n< T) \\
	&\leq \Q^n\left(\sup_{t\leq T} X_t>n \right)\\
	&\le \frac{1}{n^2} \E_{\Q^n}\left[ \sup_{t\leq T}X_t^2 \right] \\
	&\le \frac{1}{n^2} \eta(\kappa_L,T,K,G_0).
	\end{align*}
 Sending $n\to \infty$, we obtain \eqref{eq:girsanovproof1}, proving that  $M$ is  martingale.\\
	 \textit{Part 2. Martingality of $H=\exp(V^T)$.} To show that the local martinglae $H$ is a  true martingale, it is enough to bound it by a martingale, see \citet[Lemma 1.4]{jarrow2018continuous}. We fix $t\leq T$ and define  $g_i(s)=\Re(f_{i}(T-s) + \psi(T-s))\bm 1_{s\leq T}$,  and $m_i(s)=\Im(f_{i}(T-s) + \psi(T-s))\bm 1_{s\leq T}$ for $i=1,2$. Taking real parts in \eqref{eq:processweak1} yield  
	 \begin{align*}
	\Re\left(V_t^T\right)
	 &= V_0^T - \frac c 2 \int_0^t (g_1^2(s)-m^2_1(s) )dX_s +\int_0^{t} g_1(s)dM^c_s \\
	 &\quad + \int_{\R_+} \left(\cos(m_2(s)\zeta)e^{g_2(s)\zeta} -1 - g_2(s)\zeta \right) \nu(d\zeta) + \int_0^t g_2(s)dM^d_s. 
	 \end{align*}
Whence, using that  $\cos$ is bounded by $1$, we get 
\begin{align*}
 |H_t|&=\exp\left(\Re\left(V^T_t\right)\right)\leq \exp\left(V_0^T + \frac c 2 m^2_1(s) \right)  \mathcal E(U_t)
\end{align*}
with $U$ given by \eqref{eq:U}. Since, $(\psi,f_1,f_2)$ are continuous and satisfy \eqref{eq:condmomentpsi},  $m_i,g_i$ are bounded for $i=1,2$, and $g_2$ satisfies \eqref{eq:condg2}, so that $|H_t|\leq c_T \mathcal E(U_t)$ for some constant $c_T$ and $\mathcal E(U)$ is a martingale thanks to the first part. This proves that $H$ is a martingale on $[0,T]$. 
\end{proof}

\begin{lemma}\label{L:cf}
 Assume that the shifted kernels $\Delta_h K$ are in $L^2([0,T],\R)$, for all $h>0$.  Set $(G_t)_{t\geq 0}$ and $(g_t)_{t \geq 0}$ as in \eqref{eq:Gt}--\eqref{eq:gt}. 
	Then, the process $(V^T_t)_{0\le t\le T}$ defined by \eqref{eq:processweak1}-\eqref{eq:processweakbeta}-\eqref{eq:processweak2}
	satisfies
\begin{equation} \label{Y}
	\begin{aligned}
	V^T_t &=  \int_0^t f_0(T-s)dX_s + \int_0^t f_1(T-s) dM^c_s + \int_0^t f_2(T-s)dM^d_s \\
	&\quad\quad\quad\quad\quad + \int_t^T F(T-s,\psi(T-s))dG_t(s), \quad\quad t \leq T.
	\end{aligned}
	\end{equation}
\end{lemma}

\begin{proof}
	We fix $t\leq T$, $h>0$ and  we define
	\begin{align*}
	X^h_t &= G_0(t)+\int_0^t \Delta_hK(t-s)Z_s ds, \\
	g_t^h(s) &= \int_0^t \Delta_hK(s-u)dZ_u, \\
	\psi^h(t) &= \int_0^t \Delta_hK(t-s)F(s,\psi (s))ds,
	\end{align*}
	where we recall that $Z=bX + {M^c}+M^d$ and $F$ is given by \eqref{eq:RiccatiL2b}. We stress that the right-hand sides of all three quantities  are defined from $X$ and $\psi$ and do not depend on $X^h$ or $\psi^h$; $X^h_t$ and $g_t^h$ are well-defined as It\^o integrals since $\Delta_h K \in L^2([0,T],\RR)$. 	\\
	\textit{Step 1. Convergence of $X^h,g^h,\psi^h$.}
	 It  follows from the boundedness of 
	$(\psi,f_0,f_1,f_2)$ and $Z(\omega)$, condition \eqref{eq:condmomentpsi} and the $L^1$-continuity of the kernel $K$ that  
	\begin{align}\label{eq:convproof}
	\!\!\!\!\!\!\!\! \sup_{s\leq T} |\psi^h(s)-\psi(s)| \to 0, \quad  |X^h_t-X_t| \to 0, \quad \P-a.s.
	\end{align}
	as $h\to 0$. 
	Set 
	$$ G_t^h(s)=G_0(s) + \int_t^{s\vee t}{g^h_t(u)}du.$$ By invoking a stochastic Fubini theorem, justified by the $L^2$-integrability of $\Delta_h K$, we get, for all $s>t$,
	$$ G_t^h(s)-G_t(s)= \int_0^s  \left(\Delta_h K(u)-K(u) \right) \left( Z_{t\wedge (s-u)} -Z_{t-u}\right) du.$$
The boundedness of $Z(\omega)$  and the $L^1$-continuity of the kernel $K$, lead to 
\begin{align}\label{eq:convproof2}
G_t^h(s)\to G_t(s), \quad  \P-a.s.
\end{align}
as $h\to 0$, for all $s\in (t,T]$.\\
	\textit{Step 2. Proving \eqref{Y}.}
	An application of a stochastic Fubini theorem, see \citet[Theorem~2.2]{V:12} -- justified by the $L^2$-integrability of $\Delta_h K$, the boundedness of $\psi$, $f$  and $X(\omega)$ --  yields
	\begin{align}
	\int_0^t \psi^h(T-s) dZ_s &= \int_0^t \left( \int_0^{T-s} F(u,\psi(u)) \Delta_hK(T-s-u)du \right) dZ_s \\
	&= \int_0^T F(u,\psi(u)) \left( \int_0^{t\wedge (T-u)}  \Delta_hK(T-u-s)dZ_s  \right) du \\
	&= \int_0^{T-t} F(u,\psi(u)) \left( \int_0^{t}  \Delta_hK(T-u-s)dZ_s  \right) du  \\	
	&\quad\quad +  \int_{T-t}^T F(u,\psi(u)) \left( \int_0^{T-u}  \Delta_hK(T-u-s)dZ_s  \right) du \\
	&= \int_t^{T} F(T-s,\psi(T-s)) g_t^h(s) ds \\
	&\quad\quad +  \int_{0}^t F(T-s,\psi(T-s)) d\left( X^h_s -G_0(s) \right)\\
	&= \int_t^{T} F(T-s,\psi(T-s)) d\left(G_t^h(s)-G_0(s)\right) \\
	&\quad\quad +  \int_{0}^t F(T-s,\psi(T-s)) d\left( X^h_s -G_0(s) \right),
	\end{align}
	where we used in the fourth identity that $(X^h_s-G_0(s))=\int_0^s \left(\int_0^r \Delta_hK(r-u)dZ_u\right) dr$, due to Lemma~\ref{L:spotvariance} since $\Delta_h K \in L^2_{\rm loc}$, for $h>0$. 
Recalling \eqref{eq:convproof}-\eqref{eq:convproof2} and sending  $h \to 0$ in the previous identity yields, by invoking the dominated convergence for the left-hand side and Helly's second theorem on ${(X^h,G^h_t)}$ for the right-hand side (see \citet[Theorem 7.3]{natanson2016theory}), we obtain that
	\begin{align}
	\int_0^t \psi(T-s) dZ_s 
	&= \int_t^{T} F(T-s,\psi(T-s))  dG_t(s)
	+  \int_{0}^t F(T-s,\psi(T-s)) d X_s \nonumber \\
	&\quad\quad -  \int_{0}^T F(T-s,\psi(T-s)) dG_0(s). \label{eq:tempY}
	\end{align}
	Using \eqref{eq:RiccatiL2b}, we can rewrite $\beta^T$  given in \eqref{eq:processweakbeta} as 
	\begin{align*}
	\beta^T_s &=        f_0(T-s) +  b \psi(T-s)  -F(T-s,\psi(T-s)).
	\end{align*}
This	shows that $V^T$ given by \eqref{eq:processweak1} can be re-expressed in the form
	\begin{align*}
	V^T_t = V_0^T & + \int_0^t f_0(T-s)dX_s +  \int_0^t \psi(T-s)dZ_s   \\&\quad -\int_0^t F(T-s,\psi(T-s))dX_s + \int_0^t f_1(T-s) dM^c_s + \int_0^t f_2(T-s)dM^d_s.
	\end{align*}
	Plugging  \eqref{eq:tempY}   in the previous expression and recalling \eqref{eq:processweak2} yields \eqref{Y}.
\end{proof}

Combining the two previous Lemma, we prove Theorem~\ref{T:MainIntrouniqueness}.
\begin{proof}[Proof of Theorem~\ref{T:MainIntrouniqueness}] Throughout the proof we fix  $\psi \in C([0,T],\C)$  a solution to the Riccati equation \eqref{eq:RiccatiL2a}-\eqref{eq:RiccatiL2b} satisfying \eqref{eq:condmomentpsi}. We first prove \eqref{eq:affinetransform}, and then deduce the weak uniqueness statement. 
	Let  $X$ be  a non-decreasing nonnegative continuous weak solution to \eqref{eq:generalsve} for the input $(G_0,K,b,c,\nu)$.	By Lemma~\ref{L:truem} $\exp(V^T)$ is a true martingale on $[0,T]$. Its terminal value can be computed using  \eqref{Y}:
	$$ V^T_T =  \int_0^T f_0(T-s)dX_s + \int_0^T f_1(T-s) dM^c_s + \int_0^T f_2(T-s)dM^d_s.$$
	 Whence, by the martingality property, we have that 
		\begin{equation} \label{eq:expaff}
	\EE\left[ \exp\left( V^T_T\right)  \bigg| \mathcal F_t \right] = \exp\left(V^T_t\right),
	\end{equation}
	for all  $t\le T$. This proves \eqref{eq:affinetransform}. To argue uniqueness, we first observe that $V^T_0$ given in \eqref{eq:processweak2} does not depend on the process $X$, but only depends  on $G_0,K$ and $\psi$. We let $Y$ denote another   non-decreasing  nonnegative continuous weak solution to \eqref{eq:generalsve} for the  same inputs $(G_0,K,b,c,\nu)$ and we set $f_1=f_2\equiv0$. Then, \eqref{eq:expaff} holds for $Y$ with the  same function $\psi$, so that evaluating the expression at $t=0$ gives
	\begin{align}
	\E\left[ \exp\left(\int_0^T f_0(T-s)dY_s\right)\right]=\exp(V^T_0)=\E\left[ \exp\left(\int_0^T f_0(T-s)dX_s\right)\right],
	\end{align} 
	for any continuous function $f_0:[0,T]\mapsto i\R$. This yields  
that the finite-dimensional marginals $(X_{t_1}, \ldots, X_{t_m})$  and  $(Y_{t_1}, \ldots, Y_{t_m})$ are equal  for any $m$, which proves weak uniqueness.
\end{proof}

\section{Application: Hyper-rough Volterra Heston models with jumps}\label{S:Heston}

In this section, we apply our main results 
	 to a class of hyper-rough Volterra Heston models with jumps.  We fix $(G_0,K,b,c,\nu)$ as in Theorem~\ref{T:MainIntro} and we let $X$ denote the unique non-decreasing nonnegative continuous weak solution to \eqref{eq:generalsve} with $Z$ the semimartingale with characteristics  \eqref{eq:charZ} given by Theorem~\ref{T:MainIntro}. We recall the  martingales $M^c$ and $M^d$   that appear in the decomposition \eqref{eq:decompZ} of $Z$. After a possible extension of the filtered probability space, we let $M^{c,\perp}$ denote a continuous martingale independent of $M^{c}$, such that $\langle M^{c,\perp}\rangle=X$ and we set $M^S= \frac{\rho}{\sqrt c} M^c + \sqrt{1-\rho^2} M^{c,\perp}$ for some $\rho \in [-1,1]$. We  consider the following model for the log-price $S$ 
\begin{align}
\log S_t &= \log S_0 -\frac 1 2 X_t + M^S_{t},  \quad S_0>0, \label{eq:Stock}\\
X_t &= G_0(t) + \int_0^t K(t-s) Z_s  ds, \label{eq:IntegratedV}
\end{align}
where we recall the characteristics of $Z$ are $(bX,cX,\nu(d\zeta)X)$. 
For instance, if $c=0$, then $X$ can be interpreted as the `integrated intensity' of a self-exiciting jump process, e.g.~a Hawkes process, recall Example~\ref{E:hawkeschar}; if $\nu=0$, then $Z$ is continuous and $X$ can be seen as a hyper-rough process, see Remark~\ref{R:hyper} below. 

 The chief example we have in mind for $(K,G_0)$  for applications is the following: 
 \begin{example}\label{E:rough}
 \begin{itemize}
 	\item
  $K$ is proportional to the fractional kernel:
 \begin{align}\label{eq:Khyper}
 K(t)=  K_1(t) K_H(t)
 \end{align}
 where 
 $$ K_H(t)= \frac{t^{H-1/2}}{\Gamma(H+1/2)}, \quad t >0,$$
 for some $H\in(-1/2,1/2]$ and $K_1$ is a completely monotone kernel on $[0,\infty)$, e.g. $K_1\equiv 1$ or $K_1(t)=e^{-\eta t}$, for some $\eta >0$. Under such specification, the assumptions  of Theorem~\ref{T:MainIntro} needed on the kernel are satisfied due to Example~\ref{E:CM}. 
 \item
 $G_0$ is absolutely continuous:
 \begin{align}\label{eq:g0hyper}
 G_0(t)=\int_0^t g_0(s)ds, \quad t\geq 0, \quad \mbox{for some } g_0 \in \mathcal G_K,
 \end{align}
as in Example~\ref{E:g0}, recall  Remark~\ref{R:assumptiong0}.
 \end{itemize}
 \end{example}

The following remark shows that $X$ can be thought of as the `integrated variance' process in the absence of jumps.
\begin{remark}\label{R:hyper}
	Assume that $\nu=0$ and $K \in L^2_{loc}$ (e.g. $H>0$ in the specification of Example \eqref{E:rough}). It follows from  Lemma \ref{L:spotvariance}  that  $X_t=\int_0^t V_s ds$ where $(S,V)$ is a rough Volterra Heston model in the terminology of \citet[Section 7]{Abi_Jaber_2017}; \citet{El_Euch_2019} satisfying 
	\begin{align*} 
	d\log S_t &= -\frac 1 2 V_tdt + \sqrt{V_t}d\widetilde B_{t},  \quad S_0>0,\\
	V_t  &= g_0(t) + \int_0^t K(t-s) b V_s ds + \int_0^t K(t-s) \sqrt{c V_s} d\widetilde W_s,
	\end{align*}	
	for some Brownian motions $\tilde B$ and $\tilde W$ obtained from standard martingale representation theorems on a possible extension of the probability space, see for instance \citet[Proposition V.3.8]{revuz2013continuous}. For the fractional kernel with $H \in (0,1/2)$, the sample paths of $V$ are  H\"older continuous of any order strictly less than $H$ and the process $V$ is said to be `rough'.
\end{remark}

If $K_H$ is no longer in $L^2_{\rm loc}$ (e.g. $H<0$ in the specification of Example \eqref{E:rough}), not only Fubini's interchange  breaks down, but it can also be shown that $X$ is nowhere differentiable almost surely, see \citet[Proposition 4.6]{Jusselin_2018}. In this case, one cannot really make sense of the spot variance $V$ and is stuck with the `integrated variance' formulation \eqref{eq:IntegratedV}, justifying the appellation hyper--rough for such equations.

We are now in place to provide the joint Fourier--Laplace transform of $(\log S, X)$ in \eqref{eq:Stock}-\eqref{eq:IntegratedV}.  

\begin{theorem} \label{T:MainHyperrough}
	Let $(K,G_0)$ be as in Theorem~\ref{T:MainIntro} and $h_0,h_1:\R_+ \mapsto \mathbb C$ be continuous functions such that 
	\begin{align}
	\Re(h_0) \leq 0 \quad \mbox{and} \quad 0\leq \Re(h_1) \leq 1. 
	\end{align}
	The joint Fourier--Laplace transform of $( X, \log S)$ in \eqref{eq:Stock}-\eqref{eq:IntegratedV}  is given by
	$$\EE\left[ \exp\left(  \int_t^T h_0(T-u)dX_u + \int_t^T h_1(T-u)d\log S_u   \right)\!\! \Mid \!\! \Fc_t \right] \!=\! \exp\left(  \int_t^T\!  \! F(s,\psi(T-s))dG_t(s) \right)$$
	for all $t \leq T$, where $G_t$ is given by \eqref{eq:Gt} and  $\psi$ solves  the   Riccati--Volterra equation
	\begin{align}
	\psi(t)&=\int_0^t K(t-s)  F(s,\psi(s))ds, \quad t \geq 0, \label{eq:rHestonRic1} \\
	F(s,u)&= h_0(s) + \frac 12 (h^2_1(s) -h_1(s)) +  (b+\rho \sqrt{c}  h_1(s)) u + \frac {c} 2 u^2 \nonumber \\
	&\quad + \int_{\R_+}\! \left(e^{u\zeta}-1- u\zeta \right)\!\nu(d\zeta) \label{eq:rHestonRic2}
	\end{align}
\end{theorem}

\begin{proof}
	It suffices to prove that the Fourier-Laplace transform 
	\begin{align}\label{eq:Lt}
	L_t =	\EE\left[\exp\left(  \int_t^T h_0(T-u)dX_u + \int_t^T h_1(T-u)d\log S_u   \right)\Mid  \Fc_t \right] 
	\end{align}  can be written as
	\begin{align}
	L_t=\EE\left[ \exp\left( \int_t^T f_0(T-s) dX_s  + \int_t^T f_1(t-s)dM^c_s \right)\Mid  \Fc_t \right]\label{eq:tempchar}
	\end{align}
	where 
	$$  f_0(t)= h_0(t) + \frac 1 2 (h_1^2(t)-h_1(t)) -  \frac{1} 2 \rho^2 h^2_1(t) \quad \mbox{and} \quad f_1(t)=\frac{\rho}{\sqrt c} h_1(t).$$
	Indeed, if this the case, the Riccati-Volterra equations \eqref{eq:RiccatiL2a}--\eqref{eq:RiccatiL2b} reduce to \eqref{eq:rHestonRic1}-\eqref{eq:rHestonRic2} and the claimed expression for the Fourier-Laplace transform together with the existence of the corresponding solution $\psi$ follow from Theorem~\ref{T:MainIntro} (with $f_2\equiv0$), since  
	$$	\Re(f_0) + \frac c 2 \Re(f_1)^2=\Re(h_0)+ \frac 1 2 (\Re(h_1)^2 -\Re(h_1) -\Im(h_1)^2 )  \leq 0,$$
		since $\Re(f_0)\leq 0$ and  $\Re(h_1)\in [0,1]$.
	It remains to prove \eqref{eq:tempchar} by means of a projection argument. 
	For this, we fix $t\leq T$ and we write the variation of \eqref{eq:Stock} between $t$ and $T$, recall that $M^{S}=\frac{\rho}{\sqrt c} M^c + \sqrt{1-\rho^2} M^{c,\perp}$, to get
	\begin{align}\label{eq:temps}
	d\log S_u&= - \frac 1 2 dX_u  +  \rho dM^c_u + \sqrt{1-\rho^2}dM^{c,\perp}_u.
	\end{align}
	We then observe that
	\begin{align}
	\!\!\!\!\!\!\!\!\!M_t\!:&=\E\left[ \exp\left( \sqrt{1-\rho^2}\int_t^T h_1(T-s)dM^{c,\perp}_s\right)\!\!\Mid\! \Fc_t\vee \Fc^{X}\right]\\
	&= \exp\left( \frac {(1-\rho^2)} 2  \int_t^T h_1(T-s)^2 dX_s \right) \label{eq:tempM}
	\end{align}
	so that, using successively \eqref{eq:temps},  the tower property of the conditional expectation and the fact that $X$ and $Z$ are $\mathcal F^X$-measurable, $L_t$ given by \eqref{eq:Lt} satisfies 
	\begin{align*}
	L_t &= 	\EE\left[ \exp\left( \int_t^T h_0(T-u)dX_u + \int_t^T h_1(T-u)d\log S_u  \right)\Mid  \Fc_t \right] \\
	&= 	\EE\left[ \EE\left[\exp\left(  \int_t^T h_0(T-u)dX_u + \int_t^T h_1(T-u)d\log S_u   \right)\Mid \Fc_t\vee \Fc^X \right] \Mid  \Fc_t \right]\\
	&= \EE\left[ \exp\left(  \int_t^T (h_0-\frac 1 2 h_1)(T-s) dX_s  +   \int_t^T \frac{\rho}{\sqrt c} h_1(T-s)dM^{c}_u \right) M_t \Mid  \Fc_t \right]
	\end{align*}
leading to \eqref{eq:tempchar} due to \eqref{eq:tempM}. This ends the proof.
\end{proof}

In particular, we consider the specification of $G_0$ as in Example~\ref{E:rough}, and set 
$$h_0(t)\equiv u_0 \quad   \mbox{and}  \quad h_1(t)\equiv u_1, \quad \mbox{with} \quad  \Re(u_0)\leq 0,  \quad \Re(u_1) \in [0,1].$$ 
For  $t=0$ and $S_0=1$, we have $X_0=0$, $\log S_0=0$ and $dG_0(s)=g_0(s)ds$ so that the unconditional Fourier--Laplace transform reads
$$\EE\left[ \exp\left(u_0 X_T + u_1\log S_T    \right) \right] = \exp\left(    \int_0^T  F(u_1,u_2,\psi(T-s))g_0(s)ds \right), $$
with 
\begin{align*}
\psi(t)&=\int_0^t K(t-s)F(u_0,u_1,\psi(s))ds\\
	F(u_0,u_1,u_2)&= u_0 + \frac 12 (u_1^2 -u_1)  +  (b+\rho \sqrt{c}  u_1 u_2) + \frac {c} 2 u_2^2 + \int_{\R_+} \left( e^{u_2\zeta}-1-u_2\zeta \right)\nu(d\zeta). 
\end{align*}
If in addition $g_0(t) = x_0 +  \theta \int_0^t K(s) ds,$ for  some $x_0,\theta\geq 0$ (recall Example~\ref{E:g0}), then, Fubini's theorem leads to
$$ \int_0^T  F(u_0,u_1,\psi(T-s))g_0(s)ds =  x_0 \int_0^T  F(u_0,u_1,\psi(s))ds +  \theta \int_0^T \psi(s) ds $$
so that 
$$\EE\left[ \exp\left(u_0 X_T +  u_1\log S_T  \right) \right] = \exp\left(   x_0 \int_0^T  F(u_0,u_1,\psi(s))ds +  \theta \int_0^T \psi(s) ds \right). $$

\begin{remark}
	Using Theorem~\ref{T:stability}, one can prove the convergence of the  multifactor Markovian approximations designed in \citet{Abi_Jaber_2018, Abi_Jaber_2019} towards the hyper-rough Heston model, where the kernel $K_H$ is approximated by a suitable weighted sum of exponentials $K^n(t)=\sum_{i=1}^n c_i^n e^{- \gamma^n_i t}$. These  approximations are therefore still valid for non-positive values of the Hurst index $H \in (-1/2,0]$, which would allow the simulation of the process $X$ and the numerical approximation of the Riccati--Volterra equations,  we refer to the aforementioned articles for more details. 
\end{remark}

\appendix
\section{Catalytic super--Brownian motion and its local occupation time}\label{A:localtime}
In this section, we sketch a rigorous derivation of equation \eqref{eq:localtimevolterra} satisfied by the local occupation time $X$ given by \eqref{eq:localtime} formally derived in the introduction.  We will  make use  of the notation $\langle \mu, \phi \rangle$ to denote the quantity $\int_{\RR} \mu(dx) \phi(x)$.

We recall that the super--Brownian motion with a single point catalyst    $\bar Y$ solves  the  following martingale problem
\begin{align*}
\langle \bar Y_t , \phi \rangle = \langle \bar Y_0 , \phi \rangle  + \frac 12 \int_0^t \langle \bar Y_s , \Delta \phi \rangle ds  + \phi(0)Z_t,
\end{align*}
where $\Delta=\partial^2 / \partial x^2$, $\phi$ is a suitable test function and   $Z$ is a continuous martingale with quadratic variation
\begin{align*}
\langle Z \rangle_t = X_t,
\end{align*}
where $X$ is the local occupation time defined by  \eqref{eq:localtime}, 
 see \citet[Theorem 1.2.7]{dawson1994super}.
 
 In order to make the link with stochastic Volterra equations, we first reformulate the martingale problem in its `mild form'. 
 \begin{lemma}\label{L:mildlocal} Assume that $\psi \in C^2(\R,\R)$ has a Gaussian decay,  that is  $\sup_{x\in \R}|\psi(x)|e^{cz^2}<\infty$, for some constant $c$. Then, 
\begin{align*} \langle \bar Y_t, \psi \rangle  = 	\langle S_t \bar Y_0, \psi\rangle + \int_0^t (S_{t-s} \psi)(0) dZ_s,
\end{align*}
where 
 $$(S_t \mu)(x)= \int_{\RR} p_t(x-y) \mu(dy) \quad \mbox{and} \quad p_t(x)= \frac 1 {\sqrt{2\pi t}} \exp\left(-\frac{x^2}{2t}\right), \quad x \geq 0.$$
 	\end{lemma}
 
 \begin{proof}[Sketch of proof]
 	Let $\xi:\RR_+\to \RR$ be a differentiable function and  set $\phi_t(x)=\xi(t)\phi^0(x)$ for some $C^2$ function $\phi^0$ having a Gaussian decay. An application of It\^o's Lemma gives
 	\begin{align*}
 	d \langle \bar Y_t , \phi^0\rangle \xi(t) &=   \xi(t)  d \langle \bar Y_t , \psi\rangle  +   \langle \bar Y_t , \phi^0\rangle \xi'(t) dt \\
 	&=     \langle \bar Y_t , \frac 1 2 \Delta \phi_t + \partial_t \phi_t \rangle dt +  \phi_t(0)  dZ_t.
 	\end{align*}
 	Thus, 
 	\begin{align}\label{eq:appendixlocaltime}
 	\langle \bar Y_t , \phi_t \rangle  =    	 \langle \bar Y_0 , \phi_0 \rangle  + \int_0^t   \langle \bar Y_s , \frac 1 2 \Delta \phi_s + \partial_t \phi_s \rangle  ds +  \int_0^t \phi_s(0)  dZ_s.
 	\end{align}
 	Fix $t \geq 0$ and consider $\phi_s= S_{t-s}\psi$ for all $s \in [0,t]$. Noticing that $\phi_t=\psi$ and $\partial \phi_s = - \frac 1 2 \Delta \phi_s$, the claimed identity follows from \eqref{eq:appendixlocaltime} with this specific test function combined with a density argument. 
 \end{proof}

 For each $\ve >0$, let $p_{\ve}:x\to (2\pi\ve)^{-1/2}\exp(-x^2/(2\ve))$ be  Gaussian density approximations of the dirac mass at $0$. It follows from Lemma~\ref{L:mildlocal} that 
\begin{align*} \langle \bar Y_t, p^{\ve} \rangle  = 	\langle S_t \bar Y_0, p^{\ve}\rangle + \int_0^t (S_{t-s} p^{\ve})(0) dZ_s.
\end{align*}
 Integrating both sides with respect to time and invoking stochastic Fubini's theorem, see Lemma \ref{L:spotvariance}, leads to 
 \begin{align*} 
 \int_0^t \langle \bar Y_s, p^{\ve} \rangle  ds  = 	\int_0^t \langle S_s \bar Y_0, p^{\ve}\rangle ds  + \int_0^t (S_{t-s} p^{\ve})(0) Z_s ds.
 \end{align*}
Sending $\ve \to 0$  yields 
  \begin{align*} 
 X_t = \lim_{\ve \to 0}\int_0^t \langle \bar Y_s, p^{\ve} \rangle  ds  = 	\int_0^t ( S_s \bar Y_0)(0) ds  + \int_0^t p_{t-s}(0) Z_s ds,
 \end{align*}
showing that $X$ solves \eqref{eq:localtimevolterra} with the function $g_0(t)=(S_t\bar Y_0)(0)$.

\bibliographystyle{apa}
\addcontentsline{toc}{section}{References}
\bibliography{bibl}

\begin{thebibliography}{}

\bibitem[\protect\astroncite{Abi~Jaber}{2019}]{Abi_Jaber_2018}
Abi~Jaber, E. (2019).
\newblock Lifting the {H}eston model.
\newblock {\em Quantitative Finance}, 19(12):1995--2013.

\bibitem[\protect\astroncite{Abi~Jaber et~al.}{2019a}]{AJCPL:19}
Abi~Jaber, E., Cuchiero, C., Larsson, M., and Pulido, S. (2019a).
\newblock A weak solution theory for stochastic {V}olterra equations of
  convolution type.
\newblock {\em arXiv preprint arXiv:1909.01166}.

\bibitem[\protect\astroncite{Abi~Jaber and El~Euch}{2019a}]{AJEE:19markovian}
Abi~Jaber, E. and El~Euch, O. (2019a).
\newblock Markovian structure of the {V}olterra {H}eston model.
\newblock {\em Statistics \& Probability Letters}, 149:63--72.

\bibitem[\protect\astroncite{Abi~Jaber and El~Euch}{2019b}]{Abi_Jaber_2019}
Abi~Jaber, E. and El~Euch, O. (2019b).
\newblock Multifactor approximation of rough volatility models.
\newblock {\em SIAM Journal on Financial Mathematics}, 10(2):309--349.

\bibitem[\protect\astroncite{Abi~Jaber et~al.}{2019b}]{Abi_Jaber_2017}
Abi~Jaber, E., Larsson, M., Pulido, S., et~al. (2019b).
\newblock Affine {V}olterra processes.
\newblock {\em The Annals of Applied Probability}, 29(5):3155--3200.

\bibitem[\protect\astroncite{Billingsley}{1999}]{Billingsley_1999}
Billingsley, P. (1999).
\newblock {\em Convergence of Probability Measures \emph{(2nd ed.)}}.
\newblock John Wiley {\&} Sons, Inc.

\bibitem[\protect\astroncite{Brezis}{2010}]{brezis2010functional}
Brezis, H. (2010).
\newblock {\em Functional analysis, Sobolev spaces and partial differential
  equations}.
\newblock Springer Science \& Business Media.

\bibitem[\protect\astroncite{Cuchiero and
  Teichmann}{2020}]{cuchiero2020generalized}
Cuchiero, C. and Teichmann, J. (2020).
\newblock Generalized {F}eller processes and {M}arkovian lifts of stochastic
  {V}olterra processes: the affine case.
\newblock {\em Journal of Evolution Equations}, pages 1--48.

\bibitem[\protect\astroncite{Dawson and Fleischmann}{1991}]{dawson1991critical}
Dawson, D.~A. and Fleischmann, K. (1991).
\newblock Critical branching in a highly fluctuating random medium.
\newblock {\em Probability theory and related fields}, 90(2):241--274.

\bibitem[\protect\astroncite{Dawson and Fleischmann}{1994}]{dawson1994super}
Dawson, D.~A. and Fleischmann, K. (1994).
\newblock A super-{B}rownian motion with a single point catalyst.
\newblock {\em Stochastic Processes and their Applications}, 49(1):3--40.

\bibitem[\protect\astroncite{Dawson et~al.}{1995}]{dawson1995singularity}
Dawson, D.~A., Fleischmann, K., Li, Y., and Mueller, C. (1995).
\newblock Singularity of super-{B}rownian local time at a point catalyst.
\newblock {\em The Annals of Probability}, pages 37--55.

\bibitem[\protect\astroncite{{El Euch} and Rosenbaum}{2019}]{El_Euch_2019}
{El Euch}, O. and Rosenbaum, M. (2019).
\newblock {The characteristic function of rough Heston models}.
\newblock {\em Mathematical Finance}, 29(1):3--38.

\bibitem[\protect\astroncite{Etheridge}{2000}]{etheridge2000introduction}
Etheridge, A. (2000).
\newblock {\em An introduction to superprocesses}.
\newblock Number~20. American Mathematical Soc.

\bibitem[\protect\astroncite{Fleischmann and
  Le~Gall}{1995}]{fleischmann1995new}
Fleischmann, K. and Le~Gall, J.-F. (1995).
\newblock A new approach to the single point catalytic super-{B}rownian motion.
\newblock {\em Probability theory and related fields}, 102(1):63--82.

\bibitem[\protect\astroncite{Gatheral et~al.}{2018}]{Gatheral_2018}
Gatheral, J., Jaisson, T., and Rosenbaum, M. (2018).
\newblock Volatility is rough.
\newblock {\em Quantitative Finance}, 18(6):933--949.

\bibitem[\protect\astroncite{Gripenberg et~al.}{1990}]{gripenberg1990volterra}
Gripenberg, G., Londen, S.-O., and Staffans, O. (1990).
\newblock {\em Volterra integral and functional equations}, volume~34.
\newblock Cambridge University Press.

\bibitem[\protect\astroncite{Hawkes and Oakes}{1974}]{hawkes1974cluster}
Hawkes, A.~G. and Oakes, D. (1974).
\newblock A cluster process representation of a self-exciting process.
\newblock {\em Journal of Applied Probability}, 11(3):493--503.

\bibitem[\protect\astroncite{Jacod and Shiryaev}{2003}]{Jacod_2003}
Jacod, J. and Shiryaev, A.~N. (2003).
\newblock {\em Limit Theorems for Stochastic Processes \emph{(2nd ed.)}}.
\newblock Springer Berlin Heidelberg.

\bibitem[\protect\astroncite{Jarrow}{2018}]{jarrow2018continuous}
Jarrow, R.~A. (2018).
\newblock {\em Continuous-Time Asset Pricing Theory}.
\newblock Springer.

\bibitem[\protect\astroncite{Jusselin and Rosenbaum}{2018}]{Jusselin_2018}
Jusselin, P. and Rosenbaum, M. (2018).
\newblock No-arbitrage implies power-law market impact and rough volatility.
\newblock {\em e-print arXiv:1805.07134}.

\bibitem[\protect\astroncite{Kallsen}{2006}]{kallsen2006didactic}
Kallsen, J. (2006).
\newblock A didactic note on affine stochastic volatility models.
\newblock In {\em From stochastic calculus to mathematical finance}, pages
  343--368. Springer.

\bibitem[\protect\astroncite{L{\'e}pingle and
  M{\'e}min}{1978}]{lepingle1978integrabilite}
L{\'e}pingle, D. and M{\'e}min, J. (1978).
\newblock Sur l'int{\'e}grabilit{\'e} uniforme des martingales exponentielles.
\newblock {\em Zeitschrift f{\"u}r Wahrscheinlichkeitstheorie und verwandte
  Gebiete}, 42(3):175--203.

\bibitem[\protect\astroncite{Mandelbrot and
  Van~Ness}{1968}]{mandelbrot1968fractional}
Mandelbrot, B.~B. and Van~Ness, J.~W. (1968).
\newblock Fractional brownian motions, fractional noises and applications.
\newblock {\em SIAM review}, 10(4):422--437.

\bibitem[\protect\astroncite{Mytnik and Salisbury}{2015}]{mytnik2015uniqueness}
Mytnik, L. and Salisbury, T.~S. (2015).
\newblock Uniqueness for {V}olterra-type stochastic integral equations.
\newblock {\em arXiv preprint arXiv:1502.05513}.

\bibitem[\protect\astroncite{Natanson}{2016}]{natanson2016theory}
Natanson, I.~P. (2016).
\newblock {\em Theory of functions of a real variable}.
\newblock Courier Dover Publications.

\bibitem[\protect\astroncite{Perkins}{2002}]{perkins2002part}
Perkins, E. (2002).
\newblock Part ii: Dawson-{W}atanabe superprocesses and measure-valued
  diffusions.
\newblock {\em Lectures on probability theory and statistics}, pages 125--329.

\bibitem[\protect\astroncite{Revuz and Yor}{2013}]{revuz2013continuous}
Revuz, D. and Yor, M. (2013).
\newblock {\em Continuous martingales and {B}rownian motion}, volume 293.
\newblock Springer Science \& Business Media.

\bibitem[\protect\astroncite{Veraar}{2012}]{V:12}
Veraar, M. (2012).
\newblock The stochastic {F}ubini theorem revisited.
\newblock {\em Stochastics}, 84(4):543--551.

\bibitem[\protect\astroncite{Z{\"a}hle}{2005}]{zahle2005space}
Z{\"a}hle, H. (2005).
\newblock Space-time regularity of catalytic super-{B}rownian motion.
\newblock {\em Mathematische Nachrichten}, 278(7-8):942--970.

\end{thebibliography}


\end{document}